\documentclass[11pt]{article}
\usepackage{amsmath,amsfonts,amssymb,amsthm}
\usepackage{graphicx}
\usepackage{enumerate}
\usepackage{subfigure}

\newtheorem{thm}{Theorem}
\newtheorem{lemma}{Lemma}
\newtheorem{prop}{Proposition}

\newtheorem{assumption}{Assumption}

\theoremstyle{definition}
\newtheorem{example}{Example}

\usepackage[a4paper]{geometry}
\begin{document}
\title{Efficient simulation of density and probability of large deviations of sum of random vectors using saddle point representations}

\author{
  Santanu Dey\\
  \texttt{dsantanu@tcs.tifr.res.in}
\and
  Sandeep Juneja\\
  \texttt{juneja@tifr.res.in}
\and
  Ankush Agarwal\\
  \texttt{ankush@tcs.tifr.res.in}
}
\maketitle

\begin{abstract}
  We consider the problem of efficient simulation estimation of the
  density function at the tails, and the probability of large
  deviations for a sum of independent, identically distributed,
  light-tailed and non-lattice random vectors.  The latter problem
  besides being of independent interest, also forms a building block
  for more complex rare event problems that arise, for instance, in
  queuing and financial credit risk modeling.  It has been extensively
  studied in literature where state independent exponential twisting
  based importance sampling has been shown to be asymptotically
  efficient and a more nuanced state dependent exponential twisting
  has been shown to have a stronger bounded relative error property.
  We exploit the saddle-point based representations that exist for
  these rare quantities, which rely on inverting the characteristic
  functions of the underlying random vectors.  These
  representations reduce the rare event estimation problem to
  evaluating certain integrals, which may via importance sampling be
  represented as expectations.  Further, it is easy to identify and
  approximate the zero-variance importance sampling distribution to
  estimate these integrals.  We identify such importance sampling
  measures and show that they possess the asymptotically vanishing
  relative error property that is stronger than the bounded relative
  error property. To illustrate the broader applicability of the
  proposed methodology, we extend it to similarly efficiently estimate
  the practically important {\em expected overshoot} of sums of iid
  random variables.
\end{abstract}
\vspace{0.2in}

\vspace{0.2in}
\section{Introduction} 
Let $(X_i: i \geq 1)$ denote a sequence of independent, identically
distributed (iid) light tailed (their moment generating function is
finite in a neighborhood of zero) non-lattice (modulus of their
characteristic function is strictly less than one) random vectors
taking values in $\Re^d$, for $d \geq 1$. In this paper\footnote{A
  very preliminary version of this paper appeared as \cite{dey_jun}.}
we consider the problem of efficient simulation estimation of the
probability density function of $\bar{X}_n = \frac{1}{n}\sum_{i=1}^n
X_i$ at points away from $EX_i$, and the tail probability $P(\bar{X}_n
\in \mathcal{A})$ for sets $\mathcal{A}$ that do not contain $EX_i$
and essentially are affine transformations of the non-negative orthant
of $\Re^d$. We develop an efficient simulation estimation methodology
for these rare quantities that exploits the well known saddle point
representations for the probability density function of $\bar{X}_n$
obtained from Fourier inversion of the characteristic function of
$X_1$ (see e.g., \cite{butler}, \cite{daniels} and
\cite{jensen}). Furthermore, using Parseval's relation, similar
representations for $P(\bar{X}_n \in \mathcal{A})$ are easily
developed. To illustrate the broader applicability of the proposed
methodology, we also develop similar representation for $E(\bar{X}_n:
\bar{X}_n \geq a)$\footnote{Authors thank the editor for suggesting
  this application} in a single dimension setting $(d=1)$, for $a>
EX_i$, and using it develop an efficient simulation methodology for
this quantity as well.

The problem of efficient simulation estimation of the tail probability
density function has not been studied in the literature, although,
from practical viewpoint its clear that visual inspection of shape of
such density functions provides a great deal of insight into the tail
behavior of the sums of random variables. Another potential
application maybe in the maximum likelihood framework for parameter
estimation where closed form expressions for density functions of
observed outputs are not available, but simulation based estimators
provide an accurate proxy.  The problem of efficiently estimating
$P(\bar{X}_n \in \mathcal{A})$ via importance sampling, besides being
of independent importance, may also be considered a building block for
more complex problems involving many streams of i.i.d. random
variables (see e.g., \cite{parekh:walrand}, for a queuing
application; \cite{glasserman:li} for applications in credit risk
modeling). This problem has been extensively studied in rare event
simulation literature (see e.g., \cite{blan:gly},
\cite{dieker:mandjes}, \cite{glasserman:juneja},
\cite{glasserman:wang}, \cite{sadowsky}, \cite{sadowsky:bucklew}).
Essentially, the literature exploits the fact that the zero variance
importance sampling estimator for $P(\bar{X}_n \in \mathcal{A})$,
though unimplementable, has a Markovian representation. This
representation may be exploited to come up with provably efficient,
implementable approximations (see \cite{asmussen:glynn} and
\cite{juneja:shaha}).

Sadowsky and Bucklew in \cite{sadowsky:bucklew} (also see
\cite{bucklew:ney:sadowsky}) developed exponential twisting based
importance sampling algorithms to arrive at unbiased estimators for
$P(\bar{X}_n \in \mathcal{A})$ that they proved were asymptotically or
weakly efficient (as per the current standard terminology in rare
event simulation literature, see e.g., \cite{asmussen:glynn} and
\cite{juneja:shaha} for an introduction to rare event simulation.
Popular efficiency criteria for rare event estimators are also
discussed later in Section \ref{popular:measure}). The importance
sampling algorithms proposed by \cite{sadowsky:bucklew} were state
independent in that each $X_{k+1}$ was generated from a distribution
independent of the previously generated $(X_i: i \leq k)$. Blanchet,
Leder and Glynn in \cite{blan:gly} also considered the problem of
estimating $P(\bar{X}_n \in \mathcal{A})$ where they introduced state
dependent, exponential twisting based importance sampling
distributions (the distribution of generated $X_{k+1}$ depended on the
previously generated $(X_i: i \leq k)$). They showed that, when done
correctly, such an algorithm is strongly efficient, or equivalently
has the bounded relative error property.

The problem of efficient estimation of the expected overshoot $E\left
  [(\bar{X}_n-a): \bar{X}_n \geq a \right ]$ is of considerable
importance in finance and insurance settings. To the best of our
knowledge, this is the first paper that directly tackles this
estimation problem.

As mentioned earlier, in this article we exploit the saddle point
based representations of the rare event quantities considered. These
representations allow us to write the quantity of interest $\alpha_n$
as a product $c_n \times \beta_n$ where $c_n \sim \alpha_n$ (that is,
$c_n/\alpha_n \rightarrow 1$ as $n \rightarrow \infty$) and is known
in closed form. So the problem of interest is estimation of $\beta_n$,
which is an integral of a known function. Note that $\beta_n
\rightarrow 1$ as $n \rightarrow \infty$. In the literature,
asymptotic expansions for $\beta_n$ exist, however they require
computation of third and higher order derivatives of the log-moment
generating function of $X_i$. This is particularly difficult in higher
dimensions. In addition, it is difficult to control the bias in such
approximations. As we note later in numerical experiments, these
biases can be significant even when probabilities are as small as of
order $10^{-9}$. In the insurance and financial industry, simulation,
with its associated variance reduction techniques, is the preferred
method for tail risk measurement even when asymptotic approximations
are available (since these approximations are typically poor in the
range of practical interest; see e.g., \cite{glasserman:li}).

In our analysis, we note that the integral $\beta_n$ can be expressed
as an expectation of a random variable using importance
sampling. Furthermore, the zero variance estimator for this
expectation is easily ascertained. We approximate this estimator by an
implementable importance sampling distribution and prove that the
resulting unbiased estimator of $\alpha_n$ has the desirable
asymptotically vanishing relative error property. More tangibly, the
estimator of the integral $\beta_n$ has the property that its variance
converges to zero as $n \rightarrow \infty$.  An additional advantage
of the proposed approach over existing methodologies for estimating
$P(\bar{X}_n \in \mathcal{A})$ and related rare quantities is that
while these methods require $O(n)$ computational effort to generate
each sample output, our approach per sample requires small and fixed
effort independent of $n$.

The use of saddle point methods to compute tail probabilities has a
long and rich history (see e.g., \cite{butler},
\cite{lugnnani:rice} and \cite{jensen}). To the best of our knowledge
the proposed methodology is the first attempt to combine the expanding
literature on rare event simulation with the classical theory of
saddle point approximations.

The rest of the paper is organized as follows: In Section 2 we briefly
review the popular performance evaluation measures used in rare event
simulation, and the existing literature on estimating $P(\bar{X}_n \in
\mathcal{A})$.  Then, in Section 3, we develop an importance sampling
estimator for the density of $\bar{X}_n$ and show that it has
asymptotically vanishing relative error. In Section 4, we devise an
integral representation for $P(\bar{X}_n \in \mathcal{A})$ and develop
an importance sampling estimator for it and again prove that it has
asymptotically vanishing relative error.  In this section we also
discuss how this methodology can be adapted to similarly efficiently
estimate $E(\bar{X}_n: \bar{X}_n \geq a)$ in a single dimension
setting. In Section 5 we report the results of a few numerical
experiments to support our analysis. We end with a brief conclusion
and a discussion on some directions for future research in Section 6.

\section{Rare event simulation, a brief
  review}

Let $\alpha_n= E_nY_n= \int Y_n dP_n$ be a sequence of rare event
expectations in the sense that $\alpha_n \rightarrow 0$ as $ n
\rightarrow \infty$, for non-negative random variables $(Y_n: n \geq
1)$.  Here, $E_n$ is the expectation operator under $P_n$.  For
example, when $\alpha_n= P(B_n)$, $Y_n$ corresponds to the indicator
of the event $B_n$.

Naive simulation for estimating $\alpha_n$ requires generating many
iid samples of $Y_n$ under $P_n$. Their average then provides an
unbiased estimator of $\alpha_n$. Central limit theorem based
approximations then provide an asymptotically valid confidence
interval for $\alpha_n$ (under the assumption that $E_nY_n^2<\infty$).

Importance sampling involves expressing $\alpha_n= \int Y_n L_n
d\tilde{P}_n= \tilde{E}_n [Y_nL_n]$, where $\tilde{P}_n$ is another
probability measure such that $P_n$ is absolutely continuous
w.r.t. $\tilde{P}_n$, with $L_n= \frac{d P_n}{d \tilde{P}_n}$ denoting
the associated Radon-Nikodym derivative, or the likelihood ratio, and
$\tilde{E}_n$ is the expectation operator under $\tilde{P}_n$.  The
importance sampling unbiased estimator $\hat{\alpha}_n$ of $\alpha_n$
is obtained by taking an average of generated iid samples of $Y_n L_n$
under $\tilde{P}_n$.  Note that by setting
\[
d\tilde{P}_n = \frac{Y_n}{E_n(Y_n)}d P_n
\]
the simulation output $Y_n L_n$ is $E_n(Y_n)$ almost surely,
signifying that such a $\tilde{P}_n$ provides a zero variance
estimator for $\alpha_n$.

\subsection{Popular performance measures}
\label{popular:measure}
Note that the relative width of the confidence interval obtained using
the central limit theorem approximation is proportional to the ratio
of the standard deviation of the estimator divided by its mean.
Therefore, the latter is a good measure of efficiency of the
estimator.  Note that under naive simulation, when $Y_n= I(B_n)$ (For
any set $D$, $I(D)$ denotes its indicator), the standard deviation of
each sample of simulation output equals $\sqrt{\alpha_n(1-\alpha_n)}$
so that when divided by $\alpha_n$, the ratio increases to infinity as
$\alpha_n \rightarrow 0$.

Below we list some criteria that are popular in evaluating the
efficacy of the proposed importance sampling estimator (see
\cite{asmussen:glynn}).  Here, $Var(\hat{\alpha}_n)$ denotes the
variance of the estimator $\hat{\alpha}_n$ under the appropriate
importance sampling measure.

A given sequence of estimators $(\hat{\alpha}_n: n \geq 1)$ for
quantities $(\alpha_n: n \geq 1)$ is said
\begin{itemize}
\item to be {\em weakly efficient} or {\em asymptotically efficient}
  if
  \[
  \limsup_{n \rightarrow \infty} \frac{\sqrt{Var(\hat{\alpha}_n)}
  }{\alpha_n^{1-\epsilon}} < \infty
  \]
  for all $\epsilon>0$;
\item to be {\em strongly efficient } or to have {\em bounded relative
    error} if
  \[
  \limsup_{n \rightarrow \infty} \frac{\sqrt{Var(\hat{\alpha}_n)}
  }{\alpha_n} < \infty;
  \]
\item to have {\em asymptotically vanishing relative error} if
  \[
  \lim_{n \rightarrow \infty} \frac{\sqrt{Var(\hat{\alpha}_n)}
  }{\alpha_n} = 0.
  \]
\end{itemize}

\subsection{Literature review}

Recall that $(X_i: i \geq 1)$ denote a sequence of independent,
identically distributed light tailed random vectors taking values in
$\Re^d$. Let $(X_i^1, \ldots, X_i^d)$ denote the components of $X_i$,
each taking value in $\Re$.  Let $F(\cdot)$ denote the distribution
function of $X_i$.  Denote the moment generating function of $F$ by
$M(\cdot)$, so that
$$M(\theta):=E\left[e^{\theta\cdot X_1}\right]=E[e^{\theta_1X_1^1+\theta_2X_1^2+\cdots+\theta_dX_1^d}],$$
where $\theta=(\theta_1,\theta_2,\ldots,\theta_d)$ and for $x,
y\in\Re^d$ the Euclidean inner product between them is denoted by
$$x\cdot y:=x_1y_1+x_2y_2+\cdots + x_dy_d.$$
The characteristic function (CF) of $X_i$ is given by
$$\varphi(\theta):=E\left[e^{\iota\theta\cdot X_1}\right]=E[e^{\iota(\theta_1X_1^1+\theta_2X_1^2+\cdots+\theta_dX_1^d)}]$$
where $\iota=\sqrt{-1}$.  In this paper we assume that the
distribution of $X_i$ is non-lattice, which means that
$|\varphi(\theta)|<1$ for all $\theta\in\Re^d-\{0\}$.

Let $\Lambda(\theta):=\ln M(\theta)$ denote the cumulant generating
function (CGF) of $X_i$.  We define $\Theta$ to be the effective
domain of $M(\theta)$, that is
$$\Theta:=\left\{\theta=(\theta_1,\theta_2,\ldots,\theta_d)\in\Re^d|\Lambda(\theta)<\infty\right\}.$$
Throughout this article we assume that $0 \in \Theta^0$, the interior
of $\Theta$.

The large deviations rate function (see e.g., \cite{dembo}) associated
with $X_i$ is defined as
$$\Lambda^*(x) = \sup_{\theta \in\Re^d} ( \theta \cdot x - \Lambda(\theta)).$$
This can be seen to equal $\tilde{\theta}\cdot x -
\Lambda(\tilde{\theta})$ whenever there exists $\tilde{\theta}
\in\Theta^0$ such that $\Lambda'(\tilde{\theta})= x$.  (Here,
$\Lambda'$ denotes the gradient of $\Lambda$).  Now consider the
problem of estimating $P(\bar{X}_n \in \mathcal{A})$.  Let
$dF_{\theta}(x) = \exp(\theta \cdot x - \Lambda(\theta)) d F(x)$
denote the exponentially twisted distribution associated with $F$ when
the twisting parameter equals $\theta$.  Let $x_0$ denote the $\arg
\min_{x \in \mathcal{A}} \Lambda^*(x)$. Furthermore, let
$\theta^*\in\Theta^0$ solve the equation $\Lambda'(\theta)= x_0$.
Under the assumption that such a $\theta^*$ exists,
\cite{sadowsky:bucklew} propose an importance sampling measure under
which each $X_i$ is iid with the new distribution function
$F_{\theta^*}$.  Then, they prove that under this importance sampling
measure, when $\mathcal{A}$ is convex, the resulting estimator of
$P(\bar{X}_n \in \mathcal{A})$ is weakly efficient.  See
\cite{asmussen:glynn} and \cite{juneja:shaha} for a sense in which
this distribution approximates the zero variance estimator for
$P(\bar{X}_n \in \mathcal{A})$.  Since, $\Lambda'(\theta^*)= x_0$, it
is easy to see that under the exponentially twisted distribution
$F_{\theta^*}$, each $X_i$ has mean $x_0$.

As mentioned in the introduction, \cite{blan:gly} consider a variant
importance sampling measure where the distribution of $X_j$ depends on
the generated $(X_1, \ldots, X_{j-1})$.  Modulo some boundary
conditions, they choose an exponentially twisted distribution to
generate $X_j$ so that its mean under the new distribution equals
$\frac{1}{n-j+1}(nx_0 - \sum_{i=1}^{j-1}X_i)$. They prove that the
resulting estimator is strongly efficient under the restriction that
$\mathcal{A}$ is convex and has a twice continuously differentiable
boundary.  Later in Section~5, we compare the performance of the
proposed algorithm to the one based on exponential twisting developed
by \cite{sadowsky:bucklew} as well as with that proposed by
\cite{blan:gly}.

\section{Efficient estimation of probability density function of  $\bar{X}_n$} 

In this section we first develop a saddle point based representation
for the probability density function (pdf) of $\bar{X}_n$ in
Proposition~\ref{rep:density} (see e.g., \cite{butler},
\cite{daniels} and \cite{jensen}).  We then develop an approximation
to the zero variance estimator for this pdf. Our main result is
Theorem~\ref{mainthm0}, where we prove that the proposed estimator has
an asymptotically vanishing relative error.

Some notation is needed in our analysis.  Let
$$ {\Re_+^d}:=\{(x_1,x_2,\ldots,x_d)\in\Re^d|\,\,x_i\geq0\,\,\, \forall i=1,2,\ldots d\}.$$
Denote the Euclidean norm of $x\in \Re^d$ by $|x|:=\sqrt{x\cdot x}$.
For a square matrix $A$, $\text{det}(A)$ will denote the determinant
of $A$, while norm of $A$ is denoted by
$$||A||:=\max_{|x|=1} |Ax|\,.$$
Let $\Lambda''(\theta)$ denote the Hessian of $\Lambda(\theta)$ for
$\theta \in \Theta^0$.  Whenever, this is strictly positive definite,
let $A(\theta)$ be the inverse of the unique square root of
$\Lambda''(\theta)$.

\begin{prop}
  \label{rep:density}
  Suppose $\Lambda''(\theta)$ is strictly positive definite for some
  $\theta\in\Theta^0$.  Furthermore, suppose that $|\varphi|^\gamma$
  is integrable for some $\gamma\geq 1$.  Then $f_n$, the density
  function of $\bar{X}_n$, exists for all $n\geq\gamma$ and its value
  at any point $x_0$ is given by:
  \begin{equation}
    \label{density_estimate}
    f_n(x_0)=\left(\frac{n}{2\pi}\right)^{\frac{d}{2}}\frac{\exp\left[n\left\{\Lambda(\theta)-\theta\cdot x_0\right\}\right]}{\sqrt{\text{det}(\Lambda''(\theta))}}\int_{v\in \Re^d}\psi(n^{-\frac{1}{2}}A(\theta)v,\theta,n)\times\phi(v)\,dv,
  \end{equation}
  where
$$\psi(y,\theta,n)=\exp\left[n\times\eta(y,\theta)\right]$$
and
\begin{equation}
  \label{key:function}
  \eta(y,\theta)=\frac{1}{2}y^t\Lambda''(\theta)y+\Lambda\left(\theta+\iota y\right)-(\theta+\iota y)\cdot x_0-\Lambda(\theta)+\theta\cdot x_0.
\end{equation}
\end{prop}
\begin{proof}
  \begin{eqnarray}
    f_n(x_0)  &=& \left(\frac{1}{2\pi}\right)^d\int_{t\in\Re^d}M_{\bar{X}_n}(\iota t)e^{-\iota (t\cdot x_0)}\,d t\,\,\,\,[M_{\bar{X}_n}\,\,\text{is the MGF of}\,\bar{X}_n]\label{inversion:formula}\\
    &=& \left(\frac{1}{2\pi}\right)^d\int_{t\in\Re^d}M^n\left(\frac{\iota t}{n}\right)e^{-\iota (t\cdot x_0)}\,d t\,\,\,\,[M_{\bar{X}_n}\,\,\text{written in terms of}\,\,M]\nonumber\\
    &=& \left(\frac{n}{2\pi}\right)^d\int_{s\in\Re^d}M^n(\iota s)e^{-n\iota (s\cdot x_0)}\,d s\,\,\,\,[\text{substituting}\,\, s=\frac{t}{n}]\nonumber\\
    &=& \left(\frac{n}{2\pi\iota}\right)^d\int_{\theta_1-\iota\infty}^{\theta_1+\iota\infty}\int_{\theta_2-\iota\infty}^{\theta_2+\iota\infty}\cdots\int_{\theta_d-\iota\infty}^{\theta_d+\iota\infty}e^{n[\Lambda(s)-s\cdot x_0]}\,ds_1ds_2\cdots ds_d\label{cauchy}\\
    &=& \left(\frac{n}{2\pi\iota}\right)^d\int_{y\in\Re^d}\exp\left[n\left\{\Lambda(\theta+\iota y)-(\theta+\iota y)\cdot x_0\right\}\right]\,(\iota)^d d y\nonumber\\
    &=& \left(\frac{n}{2\pi}\right)^d\exp\left[n\left\{\Lambda(\theta)-\theta\cdot x_0\right\}\right]\int_{y\in\Re^d}\psi(y,\theta,n)\times \exp\left\{-n\frac{1}{2}y^t\Lambda''(\theta)y\right\}\,dy\nonumber\\
    &=& \left(\frac{n}{2\pi}\right)^{\frac{d}{2}}\exp\left[n\left\{\Lambda(\theta)-\theta\cdot x_0\right\}\right]\int_{w\in\Re^d} \psi(n^{-\frac{1}{2}}w,\theta,n)\times\phi(A(\theta)^{-1}w)\,dw\label{w_integral}\\
    &=& \left(\frac{n}{2\pi}\right)^{\frac{d}{2}}\frac{\exp\left[n\left\{\Lambda(\theta)-\theta\cdot x_0\right\}\right]}{\sqrt{\text{det}(\Lambda''(\theta))}}\int_{v\in\Re^d}\psi(n^{-\frac{1}{2}}A(\theta)v,\theta,n)\times\phi(v)\,dv\,,\label{v_integral}
  \end{eqnarray}
  where the equality in (\ref{inversion:formula}), which holds for all
  $n\geq\gamma$, is the inversion formula applied to the
  characteristic function of $\bar{X}_n$ (see e.g,
  \cite{feller2}).  The assumption that $|\varphi|^\gamma$ is
  integrable ensures that $|M(\frac{\iota t}{n})|^n$, which is the
  characteristic function of $\bar{X}_n$, is an integrable function of
  $t$ for all $n\geq\gamma$. The equality in (\ref{cauchy}) holds, by
  Cauchy's theorem, for any
  $\theta=(\theta_1,\theta_2,\ldots,\theta_d)$ in the interior of
  $\Theta$.  The substitution $y=n^{-\frac{1}{2}}w$ gives
  (\ref{w_integral}), while (\ref{v_integral}) follows from
  (\ref{w_integral}) by the substitution $w=A(\theta)v$.
\end{proof}

For a given $x_0 \in \Re^d, x_0 \ne EX_1$, suppose that the solution
$\theta^*$ to the equation $\Lambda'(\theta)= x_0$ exists and
$\theta^*\in\Theta^0$.  Then, the expansion of the integral in
(\ref{density_estimate}) is available.  For example, the following is
well-known:
\begin{prop}\label{asymptotic1}
  Suppose $\Lambda''(\theta^*)$ is strictly positive definite and
  $|\varphi|^\gamma$ is integrable for some $\gamma\geq 1$.  Then,
  \begin{equation}
    \int_{v\in\Re^d}\psi(n^{-\frac{1}{2}}A(\theta^*)v,\theta^*,n)\times\phi(v)\,dv = 1 + o\left(\frac{1}{\sqrt{n}}\right)\,.
  \end{equation}
\end{prop}

A proof of Proposition~\ref{asymptotic1} can be found in \cite{jensen}
(see also \cite{feller2}).  For completeness we include a proof in the
Appendix.  It is also useful in following proof of Proposition
\ref{asymptotic3}.  The proof uses the estimates (\ref{estimate1}),
(\ref{estimate2}), (\ref{estimate3}) and Lemma \ref{keylemma}
developed later in this section.

\subsection{Monte Carlo estimation}

The integral in (\ref{density_estimate}) may be estimated via Monte
Carlo simulation.  In particular, this integral may be re-expressed as
$$\int_{v\in\Re^d}\psi(n^{-\frac{1}{2}}A(\theta^*)v,\theta^*,n)\frac{\phi(v)}{g(v)}g(v)\,d v\,,$$
where $g$ is a density supported on $\Re^d$.  Now if
$V_1,V_2,\ldots,V_N$ are iid with distribution given by the density
$g$, then
\begin{equation}
  \label{estimator001}
  \hat{f}_n(\bar {x}):=\left(\frac{n}{2\pi}\right)^{\frac{d}{2}}\frac{\exp\left[n\left\{\Lambda(\theta^*)-\theta^*\cdot x_0\right\}\right]}{\sqrt{\text{det}(\Lambda''(\theta^*))}}\frac{1}{N}\sum_{i=1}^N\frac{\psi(n^{-\frac{1}{2}}A(\theta^*)V_i,\theta^*,n)\phi(V_i)}{g(V_i)}
\end{equation}
is an unbiased estimator for $f_n(x_0)$.

\subsubsection{Approximating the zero variance estimator\\}

Note that to get a zero variance estimator for the above integral we
need
$$g(v) \varpropto  \psi(n^{-\frac{1}{2}}A(\theta^*)v,\theta^*,n){\phi(v)}\,.$$
We now argue that
\begin{equation} \label{approx_small}
  \psi(n^{-\frac{1}{2}}A(\theta^*)v,\theta^*,n) \sim 1
\end{equation}
for all $v=o(n^{\frac{1}{6}})$. We may then select an IS density $g$
that is asymptotically similar to $\phi$ for $v=o(n^{\frac{1}{6}})$.
In the further tails, we allow $g$ to have fatter power law
tails. This ensures that large values of $V$ in the simulation do not
contribute substantially to the variance.

Further analysis is needed to see (\ref{approx_small}). Note from the
definition of $\eta(v,\theta)$, that
\begin{equation}
  \eta(0,\theta)=0, \,\,\,\,\eta''(0,\theta)=0\,\,\,\,\text{and}\,\,\,\eta'''(v ,\theta)=(\iota)^3\Lambda'''(\theta+\iota v) \label{property1}
\end{equation}
for all $\theta$, while
\begin{equation}
  \eta'(0,\theta^*)=0 \label{property2}
\end{equation}
for the saddle point $\theta^*$.  Here $\eta'$ , $\eta''$ and
$\eta'''$ are the first, second and third derivatives of $\eta$
w.r.t. $v$, with $\theta$ held fixed.  Note that while $\eta'$ and
$\eta''$ are $d$-dimensional vector and $d\times d$ matrix
respectively, $\eta'''(v,\theta)$ is the array of numbers:
$((\frac{\partial^3\eta}{\partial v_i\partial v_j\partial
  v_k}(v,\theta)))_{1\leq i,j,k\leq d}$.

The following notation aids in dealing with such quantities: If
$A=(a_{ijk})_{1\leq i,j,k\leq d}$ is a $d\times d\times d$ array of
numbers and $u=(u_1,u_2,\ldots,u_d)$ is a $d$-dimensional vector and
$B$ is a $d\times d$ matrix then we use the notation

$$A\odot u=\sum_{1\leq i,j,k\leq d}a_{ijk}u_iu_ju_k$$
and
$$A\star B=(c_{ijk})_{1\leq i,j,k\leq d}\,\,,$$
where
$$c_{ijk}=\sum_{m,n,p} a_{mnp}b_{mi}b_{nj}b_{pk}\,.$$
Following identity is evident:
\begin{equation}
  A\odot(Bu)=(A\star B)\odot u\,. \label{sily_identity}
\end{equation}
Since, it follows from the three term Taylor series expansion and
(\ref{property1},\ref{property2}) above, that
$$\psi(n^{-\frac{1}{2}}A(\theta^*)v,\theta^*,n)=\exp\left\{n \eta(n^{-\frac{1}{2}}A(\theta^*)v, \theta^*)\right\}
=\exp\left\{\frac{1}{6\sqrt{n}}\Lambda''' \left(\theta^*+ \iota
    n^{-\frac{1}{2}}A(\theta^*)\tilde{v}\right)\odot(\iota
  A(\theta^*)v)\right\}\,,$$ continuity of $\Lambda'''$ in the
neighborhood of $\theta^*$ implies (\ref{approx_small}).

\subsubsection{Proposed importance sampling density\\}
\label{proposedIS}
We now define the form of the IS density $g$.  We first show its
parametric structure and then specify how the parameters are chosen to
achieve asymptotically vanishing relative error.

For $a\in(0,\infty)$, $b\in(0,\infty)$, and $\alpha\in(1,\infty)$, set
\begin{equation}
  g(v) = \left\{
    \begin{array}{lr}
      b\times\phi(v) & \text{when}\,\, |v|<a\\
      \frac{C}{|v|^\alpha} & \text{when}\,\, |v|\geq a\,.
    \end{array}\label{ISform}
  \right.
\end{equation}
Note that if we put
$$p:=\int_{|v|<a} g(v)\,dv=b\int_{|v|<a}\phi(v)\,dv=b\times IG\left(\frac{d}{2}, \frac{a^2}{2}\right),$$
where
$$IG(\omega,x)=\frac{1}{\Gamma(\omega)}\int_0^x e^{-t}t^{\omega-1}\,dt$$
is the incomplete Gamma integral (or the Gamma distribution function,
see e.g, \cite{jensen}), then
$$C=\frac{(1-p)}{\int_{|v|\geq a}\frac{dv}{|v|^\alpha}}>0,$$
provided $p<1$.

\begin{figure}[htb]
{
\centering
\includegraphics[width=0.50\textwidth]{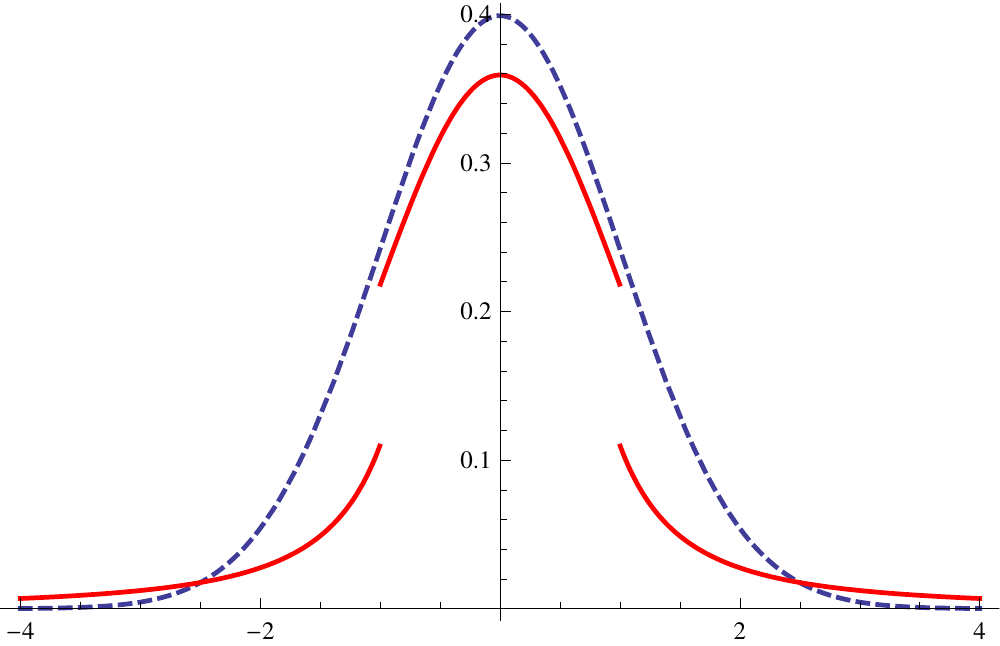}
\caption{Dotted curve is the normal density function, while solid line is the density of the proposed IS density.\label{figIS}}
}
\end{figure}

The following Assumption is important for coming up with the
parameters of the proposed IS density.
\begin{assumption}
  \label{main:assumption}
  There exist $\alpha_0>1$ and $\gamma\geq1$ such that
$$\int_{u\in\Re^d}|u|^{\alpha_0}\left|\varphi(u)\right|^\gamma\,du <\infty\,\,.$$
\end{assumption}

By Riemann-Lebesgue lemma, if the probability distribution of $X_1$ is
given by a density function, then $|\varphi(u) | \rightarrow 0$ as
$|u| \rightarrow \infty$.  Assumption~\ref{main:assumption} is easily
seen to hold when $|\varphi(u)|$ decays as a power law as $|u|
\rightarrow \infty$.  This is true, for example, for Gamma distributed
random variables.  More generally, this holds when the underlying
density has integrable higher derivatives (see \cite{feller2}): If
$k$-th order derivative of the underlying density is integrable then
for any $\alpha_0$, Assumption~\ref{main:assumption} holds with
$\gamma>\frac{1+\alpha_0}{k}$.

To specify the parameters of the IS density we need further analysis.

Define
$$\varphi_{\theta}(u):=E_{\theta}\left[e^{\iota u\cdot (X_1- x_0)}\right]
=e^{-\iota u\cdot x_0}\frac{M\left(\theta+\iota
    u\right)}{M(\theta)}\,\,,$$ where $E_{\theta}$ denotes the
expectation operator under the distribution $F_{\theta}$. Let
\begin{equation}
  \label{define:h}
  h(x):=1-\sup_{|u|\geq x}|\varphi_{\theta^*}(u)|^2.
\end{equation}

Then $0\leq h(x)\leq 1$, $h(0)=0$, $h(x)$ is continuous,
non-decreasing and $h(x)\uparrow 1$ as $x\downarrow 0$.  Further,
since $\varphi$ is the characteristic function of a non-lattice
distribution, $h(x)>0$ if $x>0$.  We define
$$h_1(y)=\min\{z\,|\,h(z)\geq y\}\,\,\text{for}\,\,y\in(0,1).$$
Then for any $y\in(0,1)$ we have $h(h_1(y))\geq y$ and
$h_1(z)\downarrow 0$ as $z\downarrow 0$.

Let $\{s_n\}_{n=1}^\infty$ be any sequence with following three
properties:
\begin{enumerate}
\item $s_n\downarrow 0$ as $n\rightarrow\infty$
\item For any $\beta$ positive, $(1-s_n)^n n^\beta\rightarrow 0$ as
  $n\rightarrow\infty$
\item $\sqrt{n}h_1(s_n)\rightarrow\infty$ as $n\rightarrow\infty$
\end{enumerate}

Later in Section 5 we discuss how such a sequence may be selected in
practice.  Set $\delta_3(n):=h_1(s_n)$.  Then, it follows that if
$x\geq\delta_3(n)$ then $h(x)\geq s_n$.  Equivalently,
$|\varphi_{\theta^*}(u)| < \sqrt{1-s_n}$ for all $u\geq\delta_3(n)$.

Let $\kappa_{min}$ and $\kappa_{max}$ denote the minimum and maximum
eigenvalue of $\Lambda''(\theta^*)$, respectively.  Hence
$\frac{1}{\kappa_{min}}$ is the maximum eigenvalue of
$\Lambda''(\theta^*)^{-1}=A(\theta^*)A(\theta^*)$.  Therefore, we have
$$\frac{1}{\kappa_{min}}=||A(\theta^*)||^2\,.$$

Next, put $\delta_2(n)=\sqrt{\kappa_{max}}\delta_3(n)$.  Then,
$\sqrt{n}\delta_2(n)\rightarrow \infty$ and $|v|\geq\delta_2(n)$
implies $|A(\theta^*)v|\geq\delta_3(n)$.  Also let
$$\delta_1(n)=\frac{1}{\sqrt{\kappa_{min}}}\delta_2(n)=\sqrt{\frac{\kappa_{max}}{\kappa_{min}}}\delta_3(n),$$
so that $|v|<\delta_2(n)$ implies $|A(\theta^*)v|<\delta_1(n)$.

Now we are in position to specify the parameters for the proposed IS
density.  Set
$$\alpha=\alpha_0$$
and
$$a_n=\sqrt{n}\delta_2(n).$$
Let $p_n=b_n\times IG\left(\frac{d}{2},\frac{a_n^2}{2}\right)$.  For
$g$ to be a valid density function, we need $p_n<1$.  Since
$IG\left(\frac{d}{2},\frac{a_n^2}{2}\right)\rightarrow 1$, select
$b_n$ to be a sequence of positive real numbers that converge to 1 in
such a way that $b_n<1/IG\left(\frac{d}{2},\frac{a_n^2}{2}\right)$ and
\begin{equation}
  \label{lim999}
  \lim_{n\rightarrow\infty}\frac{(1-s_n)^{n}n^{\frac{d+\alpha}{2}}}{\left[1-b_n\times IG\left(\frac{d}{2},\frac{a_n^2}{2}\right)\right]}=0.
\end{equation}
For example, $b_n = 1 - n^{-\xi}$ for any $\xi > 0$ satisfies
(\ref{lim999}).  For each $n$, let $g_n$ denote the pdf of the form
(\ref{ISform}) with parameters $\alpha$, $a_n$ and $b_n$ chosen as
above.  Let $E_n$ and $Var_n$ denote the expectation and variance,
respectively, w.r.t. the density $g_n$.

\begin{thm}
  \label{mainthm0}
  Suppose Assumption \ref{main:assumption} holds and
  $\theta^*\in\Theta^0$.  Then,
  \begin{equation*}
    E_n\left[\frac{\psi^2(n^{-\frac{1}{2}}A(\theta^*)V,\theta^*,n)\phi^2(V)}{g^2_n(V)}\right]=\int_{v\in\Re^d}\frac{\psi^2(n^{-\frac{1}{2}}A(\theta^*)v,\theta^*,n)\phi^2(v)}{g_n(v)}\,dv=1+o(n^{-\frac{1}{2}})\,.
  \end{equation*}
  Consequently, from Proposition \ref{asymptotic1}, it follows that
$$Var_n\left[\frac{\psi(n^{-\frac{1}{2}}A(\theta^*)V_i,\theta^*,n)\phi(V_i)}{g_n(V_i)}\right]\rightarrow 0\,\,\,\text{as}\,\,n\rightarrow \infty\,,$$
so that the proposed estimators for $(f_n(x_0): n\geq 1)$ have an
asymptotically vanishing relative error.
\end{thm}

We will use the following lemma from \cite{feller2}.
\begin{lemma}
  \label{keylemma}
  For any $\lambda,\,\beta\in\mathbb{C}$,
 $$|\exp(\lambda)-1-\beta|\leq \left(|\lambda-\beta|+\frac{|\beta|^2}{2}\right)\exp(\omega)\,\,\,\,\text{for all}\,\,\omega\geq\max\{|\lambda|,|\beta|\}\,.$$
\end{lemma}

Also note that from the definitions of $\psi$ and $\eta$ it follows
that, for any $\theta\in\Theta$,
$$\exp\left\{-\frac{v\cdot v}{2}\right\}\psi(n^{-\frac{1}{2}}A(\theta)v,\theta,n)$$
is a characteristic function. To see this, observe that
\begin{eqnarray*}
  \exp\left\{-\frac{v\cdot v}{2}\right\}\psi(n^{-\frac{1}{2}}A(\theta)v,\theta,n)&=&\left[\exp\left\{-\frac{v\cdot v}{2n}+\eta\left(n^{-\frac{1}{2}}A(\theta)v,\theta\right)\right\}\right]^n\\
  &=&\left(E_{\theta}\left[e^{\iota n^{-\frac{1}{2}}A(\theta)v\cdot (X_1- x_0)}\right]\right)^n\\
  &=&\left[\varphi_{\theta}\left(n^{-\frac{1}{2}}A(\theta)v\right)\right]^n.
\end{eqnarray*}

Some more observations are useful for proving Theorem \ref{mainthm0}.

Since $\eta'''$ is continuous, it follows from the three term Taylor
series expansion,
$$\eta(v,\theta)=\eta(0 ,\theta)+\eta'(0 ,\theta)v+\frac{1}{2}(v)^T\eta''(0,\theta)v+\frac{1}{6}\eta'''(\tilde{v} ,\theta)\odot v$$
(where $\tilde{v}$ is between $v$ and the origin) and
(\ref{property1}) and (\ref{property2}) above that there exists a
sequence $\{\epsilon_n\}$ of positive numbers converging to zero so
that
$$|\eta(v,\theta^*)-\frac{1}{3!}\eta'''(0,\theta^*)\odot v|\leq \epsilon_n (\kappa_{min})^{\frac{3}{2}}|v|^3\,\,\,\,\text{for}\,\,|v|<\delta_1(n),$$
or equivalently
\begin{equation}
  \label{estimate122}
  |\eta(v,\theta^*)-\frac{1}{3!}\Lambda'''(\theta^*)\odot (\iota v)|\leq \epsilon_n(\kappa_{min})^{\frac{3}{2}}|v|^3\,\,\,\,\text{for}\,\,|v|<\delta_1(n)\,.
\end{equation}

Furthermore, for $n$ sufficiently large,
\begin{equation}
  \label{estimate222}
  \left|\frac{1}{3!}\Lambda'''(\theta^*)\odot (\iota v)\right|<\frac{1}{8}\kappa_{min}|v|^2
\end{equation}
and
\begin{equation}
  \label{estimate322}
  |\eta(v,\theta^*)|<\frac{1}{8}\kappa_{min}|v|^2
\end{equation}
for all $|v|<\delta_1(n)$. We shall assume that $n$ is sufficiently
large so that (\ref{estimate222}) and (\ref{estimate322}) hold in the
remaining analysis.

\begin{proof} ( {\bf Theorem \ref{mainthm0}})\\
  We
  write $$\int_{v\in\Re^d}\frac{\psi^2(n^{-\frac{1}{2}}A(\theta^*)v,\theta^*,n)\phi^2(v)}{g_n(v)}\,dv=I_3+I_4\,.$$
  Where
$$I_3=\int_{|v|<\sqrt{n}\delta_2(n)}\frac{\psi^2(n^{-\frac{1}{2}}A(\theta^*)v,\theta^*,n)\phi^2(v)}{g_n(v)}\,dv$$
and
$$I_4=\int_{|v|\geq\sqrt{n}\delta_2(n) }\frac{\psi^2(n^{-\frac{1}{2}}A(\theta^*)v,\theta^*,n)\phi^2(v)}{g_n(v)}\,dv.$$
From (\ref{ISform}) we get
$$I_3=\frac{1}{b_n}\int_{|v|<\sqrt{n}\delta_2(n)}\psi^2(n^{-\frac{1}{2}}A(\theta^*)v,\theta^*,n)\phi(v)\,dv$$
and
$$I_4=\frac{1}{C_{n}}\int_{|v|\geq\sqrt{n}\delta_2(n)}|v|^\alpha\psi^2(n^{-\frac{1}{2}}A(\theta^*)v,\theta^*,n)\phi^2(v)\,dv.$$
For any $c>0$, put
$$\Phi_d(c):=\int_{|v|<c}\phi(v)dv \left(=IG\left(\frac{d}{2}, \frac{c^2}{2}\right)\right)\,.$$
By triangle inequality we have
$$|I_3-1|\leq\left|I_3-\frac{\Phi_d\left(\sqrt{n}\delta_2(n)\right)}{b_n}\right|+\left|\frac{\Phi_d\left(\sqrt{n}\delta_2(n)\right)}{b_n}-1\right|.$$
Since as $n\rightarrow\infty$ we have
$\Phi_d\left(\sqrt{n}\delta_2(n)\right)\rightarrow 1$ and
$b_n\rightarrow 1$, the second term in RHS converges to zero.  Writing
$\zeta_3(\theta^*)=\Lambda'''(\theta^*)\star A(\theta^*)$, for the
first term we have
\begin{eqnarray*}
  \left|I_3-\frac{\Phi_d\left(\sqrt{n}\delta_2(n)\right)}{b_n}\right|&=&\frac{1}{b_n}\left|\int_{|v|<\sqrt{n}\delta_2(n)}\left\{\psi^2(n^{-\frac{1}{2}}A(\theta^*)v,\theta^*,n)-1\right\}\phi(v)\,dv\right|\nonumber\\
  &=&\frac{1}{b_n}\left|\int_{|v|<\sqrt{n}\delta_2(n)}\left\{\psi^2(n^{-\frac{1}{2}}A(\theta^*)v,\theta^*,n)-1-\frac{\zeta_3(\theta^*)}{3\sqrt{n}}\odot(\iota v)\right\}\phi(v)\,dv\right|\nonumber\\
  &\leq&\frac{1}{b_n}\frac{1}{(2\pi)^{\frac{d}{2}}}\int_{|v|<\sqrt{n}\delta_2(n)}\left|\psi^2(n^{-\frac{1}{2}}A(\theta^*)v,\theta^*,n)-1-\frac{\zeta_3(\theta^*)}{3\sqrt{n}}\odot(\iota v)\right|e^{-\frac{v^2}{2}}\,dv.
\end{eqnarray*}
We apply Lemma (\ref{keylemma}) with
$$\lambda=2n\times\eta\left(n^{-\frac{1}{2}}A(\theta^*)v,\theta^*\right)\,\,\,\text{and}\,\,\,\beta=n\frac{\Lambda'''(\theta^*)}{3}\odot\left(\iota n^{-\frac{1}{2}}A(\theta^*)v\right).$$
Since $\frac{|\beta|^2}{2}=\frac{1}{n}P(v)$, where $P$ is a
homogeneous polynomial whose coefficients does not dependent on $n$,
and $|v|<\sqrt{n}\delta_2(n)$ implies
$|n^{-\frac{1}{2}}A(\theta^*)v|<\delta_1(n)$, we have from
(\ref{estimate322}), (\ref{estimate222}) and (\ref{estimate122}),
respectively

$$|\lambda|=2n\left|\eta\left(n^{-\frac{1}{2}}A(\theta^*)v,\theta^*\right)\right|<2n\frac{1}{8}\kappa_{min}|n^{-\frac{1}{2}}A(\theta^*)v|^2\leq\frac{1}{8}\kappa_{min}||A(\theta^*)||^2|v|^2=\frac{|v|^2}{4},$$
$$|\beta|=2n\left|\frac{1}{3!}\Lambda'''(\theta^*)\odot\left(\iota n^{-\frac{1}{2}}A(\theta^*)v\right)\right|<2n\frac{1}{8}\kappa_{min}|n^{-\frac{1}{2}}A(\theta^*)v|^2\leq\frac{1}{8}\kappa_{min}||A(\theta^*)||^2|v|^2=\frac{|v|^2}{4}$$
and
$$|\lambda-\beta|=2n\left|\eta\left(n^{-\frac{1}{2}}A(\theta^*)v,\theta^*\right)-\frac{1}{3!}\Lambda'''(\theta^*)\odot\left(\iota n^{-\frac{1}{2}}A(\theta^*)v\right)\right|<2n\epsilon_n(\kappa_{min})^{\frac{3}{2}}|n^{-\frac{1}{2}}A(\theta^*)v|^3\leq\frac{2\epsilon_n|v|^3}{\sqrt{n}}.$$
From Lemma~\ref{keylemma}, it now follows that the integrand in the
last integral is dominated by
$$\exp\left\{\frac{|v|^2}{4}\right\}\times\left(\frac{2\epsilon_n|v|^3}{\sqrt{n}}+\frac{1}{n}P(v)\right)\exp\left\{-\frac{|v|^2}{2}\right\}\times=
\exp\left\{-\frac{|v|^2}{4}\right\}\left(\frac{2\epsilon_n|v|^3}{\sqrt{n}}+\frac{1}{n}P(v)\right).$$
Therefore we have $I_3= 1 + o(n^{-\frac{1}{2}})$.

Also
\begin{eqnarray*}
  |I_4|&\leq&\frac{1}{(2\pi)^d C_{n}}\int_{|v|>\sqrt{n}\delta_2(n)}|v|^\alpha\left|\exp\left\{-|v|^2\right\}\psi^2(n^{-\frac{1}{2}}A(\theta^*)v,\theta^*,n)\right|\,dv\\
  &=&\frac{1}{(2\pi)^d C_{n}}\int_{|v|>\sqrt{n}\delta_2(n)}|v|^\alpha\left|\varphi_{\theta^*}\left(n^{-\frac{1}{2}}A(\theta^*)v\right)\right|^{2n}\,dv\\
  &\leq&\frac{(1-s_n)^{n-\frac{\gamma}{2}}}{(2\pi)^d C_{n}}\int_{v\in\Re}|v|^\alpha\left|\varphi_{\theta^*}\left(n^{-\frac{1}{2}}A(\theta^*)v\right)\right|^\gamma\,dv\label{eqn323}\\
  &=& \frac{(1-s_n)^{n-\frac{\gamma}{2}}n^{\frac{d+\alpha}{2}}\sqrt{|\Lambda''(\theta^*)|}}{(2\pi)^d C_{n}}\int_{u\in\Re}|A(\theta^*)^{-1}u|^\alpha\left|\varphi_{\theta^*}(u)\right|^\gamma\,du\\
  &\leq& D_1\frac{(1-s_n)^{n-\frac{\gamma}{2}}n^{\frac{d+\alpha}{2}}}{C_{n}}\int_{u\in\Re}|u|^\alpha\left|\varphi_{\theta^*}(u)\right|^\gamma\,du\\
  &\leq& D_1\frac{(1-s_n)^{n-\frac{\gamma}{2}}n^{\frac{d+\alpha}{2}}\int_{|v|\geq \sqrt{n}\delta_2(n)}\frac{dv}{|v|^\alpha}}{(1-p_n)}\int_{u\in\Re}|u|^\alpha\left|\varphi_{\theta^*}(u)\right|^\gamma\,du\,.
\end{eqnarray*}
where $D_1$ is a constant independent of $n$.  By Assumption
\ref{main:assumption}, the above integral over $u$ is finite.  For
large $n$ we also have
$$\int_{|v|\geq \sqrt{n}\delta_2(n)}\frac{dv}{|v|^\alpha}\leq\int_{|v|\geq 1}\frac{dv}{|v|^\alpha}.$$
By choice of $b_n$ we can conclude that $I_4\rightarrow0$ as
$n\rightarrow\infty$, proving Theorem \ref{mainthm0}.
\end{proof}

\section{Efficient Estimation of Tail Probability}
In this section we consider the problem of efficient estimation of
$P(\bar{X}_n \in \mathcal{A})$ for sets $\mathcal{A}$ that are affine
transformations of the non-negative orthants $\Re^d_+$ along with some
minor variations.  As in (\cite{bucklew}), dominating point of the set
$\mathcal{A}$ plays a crucial role in our analysis.  As is well known,
a point $x_0$ is called a dominating point of $\mathcal{A}$ if $x_0$
uniquely satisfies the following properties (see e.g, \cite{ney},
\cite{bucklew}):
\begin{enumerate}
\item $x_0$ is in the boundary of $\mathcal{A}$.
\item There exists a unique $\theta^*\in\Re^d$ with
  $\Lambda^{\prime}(\theta^*)=x_0$.
\item $\mathcal{A}\subseteq\{x|\theta^*\cdot(x-x_0) \geq 0\}$.
\end{enumerate}

As is apparent from (\cite{ney}, \cite{sadowsky:bucklew},
\cite{bucklew}), in many cases a general set $\mathcal{A}$ may be
partitioned into finitely many sets $(\mathcal{A}_i : i \leq m)$ each
having its own dominating point.  From simulation viewpoint, one way
to estimate $P(\bar{X}_n \in \mathcal{A})$ then is to estimate each
$P(\bar{X}_n \in \mathcal{A}_i)$ separately with an appropriate
algorithm.  In the remaining paper, we assume the existence of a
dominating point $x_0$ for $\mathcal{A}$.

Our estimation relies on a saddle-point representation of $P(\bar{X}_n
\in \mathcal{A})$ obtained using Parseval's relation. Let
$$Y_n :=\sqrt{n}(\bar{X}_n - x_0)$$
and
$$\mathcal{A}_{n,x_0} := \sqrt{n}(\mathcal{A} - x_0)$$
where $x_0=(x_0^1,x_0^2,\ldots,x_0^d)$ is an arbitrarily chosen point
in $\Re^d$.  Let $h_{n,\theta,x_0}(y)$ be the density function of
$Y_n$ when each $X_i$ has distribution function $F_{\theta}$, where,
recall that
$$dF_{\theta}(x) = \exp(\theta\cdot x)M(\theta)^{-1}dF(x) = \exp\{\theta\cdot x - \Lambda(\theta)\}dF(x)\,.$$
An exact expression for the tail probability is given by:
\begin{eqnarray}
  \label{tail:generalized}
  P[\bar{X}_n\in \mathcal{A}]= P[Y_n\in\mathcal{A}_{n,x_0}] = e^{-n\{\theta\cdot x_0 - \Lambda(\theta)\}}\int_{y\in\mathcal{A}_{n,x_0}}e^{-\sqrt{n}(\theta\cdot y)}h_{n,\theta, x_0}(y)\,dy\,
\end{eqnarray}
which holds for any $\theta\in\Theta$ and any $x_0\in\Re^d$.  The
representation (\ref{tail:generalized}) is not very useful without
further restriction on $x_0$ and $\theta$ (see e.g., \cite{ney}).
Again, assuming that a solution $\theta^*\in\Theta^0$ to
$\Lambda'(\theta) = x_0$ exists, where $x_0$ is the dominating point
of $\mathcal{A}$, define
$$c(n,\theta^*,x_0) = \int_{y\in\mathcal{A}_{n,x_0}} \exp\{-\sqrt{n}(\theta^*\cdot y)\}\,dy=n^{\frac{d}{2}}\int_{w\in(\mathcal{A}-x_0) } \exp\{-n(\theta^*\cdot w)\}\,dw$$
We need the following assumption:
\begin{assumption}
  \label{cn:finite}
  $\forall n$, $c(n,\theta^*,x_0)<\infty$.
\end{assumption}

Since $x_0$ is a dominating point of $\mathcal{A}$, for any
$y\in\mathcal{A}_{n,x_0}$, we have $\theta^*\cdot y\geq 0$.  Hence, if
$\mathcal{A}$ is a set with finite Lebesgue measure then
$c(n,\theta^*,x_0)$ is finite.  Assumption \ref{cn:finite} may hold
even when $\mathcal{A}$ has infinite Lebesgue measure, as
Example~\ref{positive:orthant} below illustrates.

When Assumption \ref{cn:finite} holds, we can rewrite the right hand
side of (\ref{tail:generalized}) as
\begin{equation}
  \label{pre:parseval}
  c(n,\theta^*,x_0) e^{-n\{\theta^*\cdot x_0 - \Lambda(\theta^*)\}}\int_{y\in\mathcal{A}_{n,x_0}} r_{n,\theta^*,x_0}(y)h_{n,\theta^*,x_0}(y)\,dy\,
\end{equation}
where
\begin{equation}
  r_{n,\theta^*,x_0}(y) = \left\{
    \begin{array}{lr}
      \frac{\exp\{-\sqrt{n}(\theta^*\cdot y)\}}{c(n,\theta^*,x_0)} & \text{when}\,\, y\in\mathcal{A}_{n,x_0}\\
      0 & \text{otherwise}
    \end{array}\label{exp:like}
  \right.
\end{equation}
is a density in $\Re^d$.

Let $\rho_{n,\theta^*,x_0}(t)$ denote the complex conjugate of the
characteristic function of $r_{n,\theta^*,x_0}(y)$.  Since the
characteristic function of $h(n,\theta^*,x_0)$ equals
$$e^{-\iota t \sqrt{n} x_0}\left[\frac{M\left(\theta^* + \frac{\iota t}{\sqrt{n}}\right)}{M(\theta^*)}\right]^n,$$
by Parseval's relation, (\ref{pre:parseval}) is equal to
\begin{equation}
  \label{post:parseval}
  c(n,\theta^*,x_0) e^{-n\{\theta^*\cdot x_0 - \Lambda(\theta^*)\}}\left(\frac{1}{2\pi}\right)^d\int_{t\in\Re^d}\rho_{n,\theta^*,x_0}(t)e^{-\iota t \sqrt{n} x_0}\left[\frac{M\left(\theta^* + \frac{\iota t}{\sqrt{n}}\right)}{M(\theta^*)}\right]^n\,dt.
\end{equation}

This in turn, by the change of variable $t=A(\theta^*)v$ and
rearrangement of terms, equals
\begin{equation}
  \label{final}
  \frac{c(n,\theta^*,x_0) e^{-n\{\theta^*\cdot x_0 - \Lambda(\theta^*)\}}}{\sqrt{\operatorname{det}(\Lambda^{\prime\prime}(\theta^{*}))}}\left(\frac{1}{2\pi}\right)^{\frac{d}{2}} \int_{v\in\Re^d}\rho_{n,\theta^*,x_0}(A(\theta^*)v)\psi(n^{-\frac{1}{2}}A(\theta^*)v,\theta^*,n)\phi(v)\,dv.
\end{equation}
We need another assumption to facilitate analysis:
\begin{assumption}
  \label{limit:chfn}
  For all $t\in\Re^d$,
$$\lim_{n\rightarrow\infty}\rho_{n,\theta^*,x_0}(t) = 1.$$
\end{assumption}

\begin{prop}
  \label{asymptotic3}
  Suppose $\mathcal{A}$ has a dominating point $x_0$, the associated
  $\theta^*\in\Theta^o$ and $\Lambda''(\theta^*)$ is strictly positive
  definite.  Further, Assumptions \ref{cn:finite} and \ref{limit:chfn}
  hold.  Then,
  \begin{equation}
    \label{limit:007}
    P[\bar{X}_n\in \mathcal{A}]\sim \left(\frac{1}{2\pi}\right)^{\frac{d}{2}} \frac{c(n,\theta^*,x_0) e^{-n\{\theta^*\cdot x_0 - \Lambda(\theta^*)\}}}{\sqrt{\operatorname{det}(\Lambda^{\prime\prime}(\theta^{*}))}},
  \end{equation}
  or, equivalently by (\ref{final})
  \begin{equation}
    \label{limit:008}
    \lim_{n\rightarrow\infty} \int_{v\in\Re^d}\rho_{n,\theta^*,x_0}(A(\theta^*)v)\psi(n^{-\frac{1}{2}}A(\theta^*)v,\theta^*,n)\phi(v)\,dv= 1.
  \end{equation}
\end{prop}
Proof of Proposition \ref{asymptotic3} is omitted.  It follows along
the line of proof of Proposition \ref{asymptotic1} and from noting
that:
$$\lim_{n\rightarrow\infty} \int_{v\in\Re^d}\rho_{n,\theta^*,x_0}(A(\theta^*)v)\phi(v)\,dv = 1,$$
$$\lim_{n\rightarrow\infty} \int_{v\in\Re^d}v_iv_jv_k \rho_{n,\theta^*,x_0}(A(\theta^*)v)\phi(v)\,dv = 0.$$

Let $g$ be any density supported on $\Re^d$.  If $V_1,V_2,\ldots,V_N$
are iid with distribution given by density $g$, then the unbiased
estimator for $P[\bar {X}_n\in \mathcal{A}]$ is given by
\begin{eqnarray}
  \label{tail:estimator}
  \hat {P}[\bar{X}_n\in \mathcal{A}] &=& \left(\frac{1}{2\pi}\right)^{\frac{d}{2}}\frac{c(n,\theta^*,x_0) e^{-n\{\theta^*\cdot x_0 - \Lambda(\theta^*)\}}}{\sqrt{\operatorname{det}(\Lambda^{\prime\prime}(\theta^{*}))}}\nonumber\\
  & & \times\frac{1}{N}\sum_{j=1}^N \frac{\rho_{n,\theta^*,x_0}(A(\theta^*)V_j)\psi(n^{-\frac{1}{2}}A(\theta^*)V_j,\theta^*,n)\phi(V_j)}{g(V_j)}.
\end{eqnarray}
Note that for above estimator to be useful, one must be able to find
closed form expression for $c(n,\theta^*,x_0)$ and
$\rho_{n,\theta^*,x_0}(t)$ or these should be cheaply computable.  In
Section \ref{examples}, we consider some examples where we explicitly
compute $c(n,\theta^*,x_0)$ and $\rho_{n,\theta^*,x_0}$ and verify
Assumptions \ref{cn:finite} and \ref{limit:chfn}.

\begin{thm}
  \label{mainthm2}
  Under Assumptions \ref{main:assumption}, \ref{cn:finite} and
  \ref{limit:chfn},
  \begin{equation*}
    E_n\left[\frac{\rho_{n,\theta^*,x_0}^2(A(\theta^*)V)\psi^2(n^{-\frac{1}{2}}A(\theta^*)V,\theta^*,n)\phi^2(V)}{g_n^2(V)}\right]=
    1 + o(n^{-\frac{1}{2}})\,\,\,\,\text{as}\,\,\,n\rightarrow\infty,
  \end{equation*}
  where $g_n$ is same as Theorem \ref{mainthm0}.  Consequently, by
  Proposition \ref{asymptotic3}, it follows that as
  $n\rightarrow\infty$
$$Var_n\left[\hat {P}[\bar{X}_n\in \mathcal{A}]\right]\rightarrow 0$$
and the proposed estimator has asymptotically vanishing relative
error.
\end{thm}

The proof of Theorem~\ref{mainthm2} is given in the appendix.

\subsection{Examples}
\label{examples}
\begin{example}
  \label{positive:orthant}
  Let $\mathcal{A} = x_0 + {\Re_+^d}$, where
  $x_0=(x_0^1,x_0^2,\ldots,x_0^d)$ is a given point in $\Re^d$.
  Further suppose that $\forall i=1,2,\ldots,d,\,\,\theta_i^*>0$.  It
  is easy to see that existence of such a $\theta^*$ implies that
  $x_0$ is a dominating point for $\mathcal{A}$.  It also follows that
  Assumption~\ref{cn:finite} holds and
$$c(n,\theta^*,x_0)=\frac{1}{n^{\frac{d}{2}}\theta_1^*\theta_2^*\cdots\theta_d^*}.$$
It can easily be verified that
$$\rho_{n,\theta^*,x_0}(t_1,t_2,\ldots t_d)=\prod_{i=1}^{d}\left(\frac{1}{1+\frac{\iota t_i}{\sqrt{n}\theta_i^*}}\right).$$
Therefore Assumption~\ref{limit:chfn} also holds in this case.  By
Proposition \ref{asymptotic3}, we then have
$$P[\bar{X}_n-x_0\in\Re_+^d]\sim \frac{e^{n\left\{\Lambda(\theta^*)-\theta^*\cdot x_0\right\}}}{(2\pi)^{\frac{d}{2}} n^{\frac{d}{2}} \sqrt{\operatorname{det}(\Lambda^{\prime\prime}(\theta^{*}))} \theta_1^*\theta_2^*\cdots\theta_d^*}.$$
By Theorem~\ref{mainthm2},
\begin{equation}
  \label{estimator002}
  \hat {P}[\bar{X}_n-x_0\in\Re_+^d]:=\frac{e^{n\left\{\Lambda(\theta^*)-\theta^* \cdot x_0\right\}}}{(2\pi)^{\frac{d}{2}} n^{\frac{d}{2}}\sqrt{\operatorname{det}(\Lambda^{\prime\prime}(\theta^{*}))}\theta_1^*\theta_2^*\cdots\theta_{d}^*}
  \times\frac{1}{N}\sum_{j=1}^N\frac{\psi(n^{-\frac{1}{2}}A(\theta^*)V_j,\theta^*,n)\phi(V_j)}{\prod_{i=1}^{d}\left(1+\frac{\iota e_i^TA(\theta^*)V_j}{\sqrt{n}\theta_i^*}\right)g(V_j)}
\end{equation}
is an unbiased estimator for $P[\bar {X}_n-x_0\in\Re_+^d]$ and has an
asymptotically vanishing relative error.
\end{example}

\begin{figure}[htb]
  \begin{center}
    \subfigure[\label{fig1}$\mathcal{A}=x_0 + \Re_+^d$.]{
      \includegraphics[width=6cm]{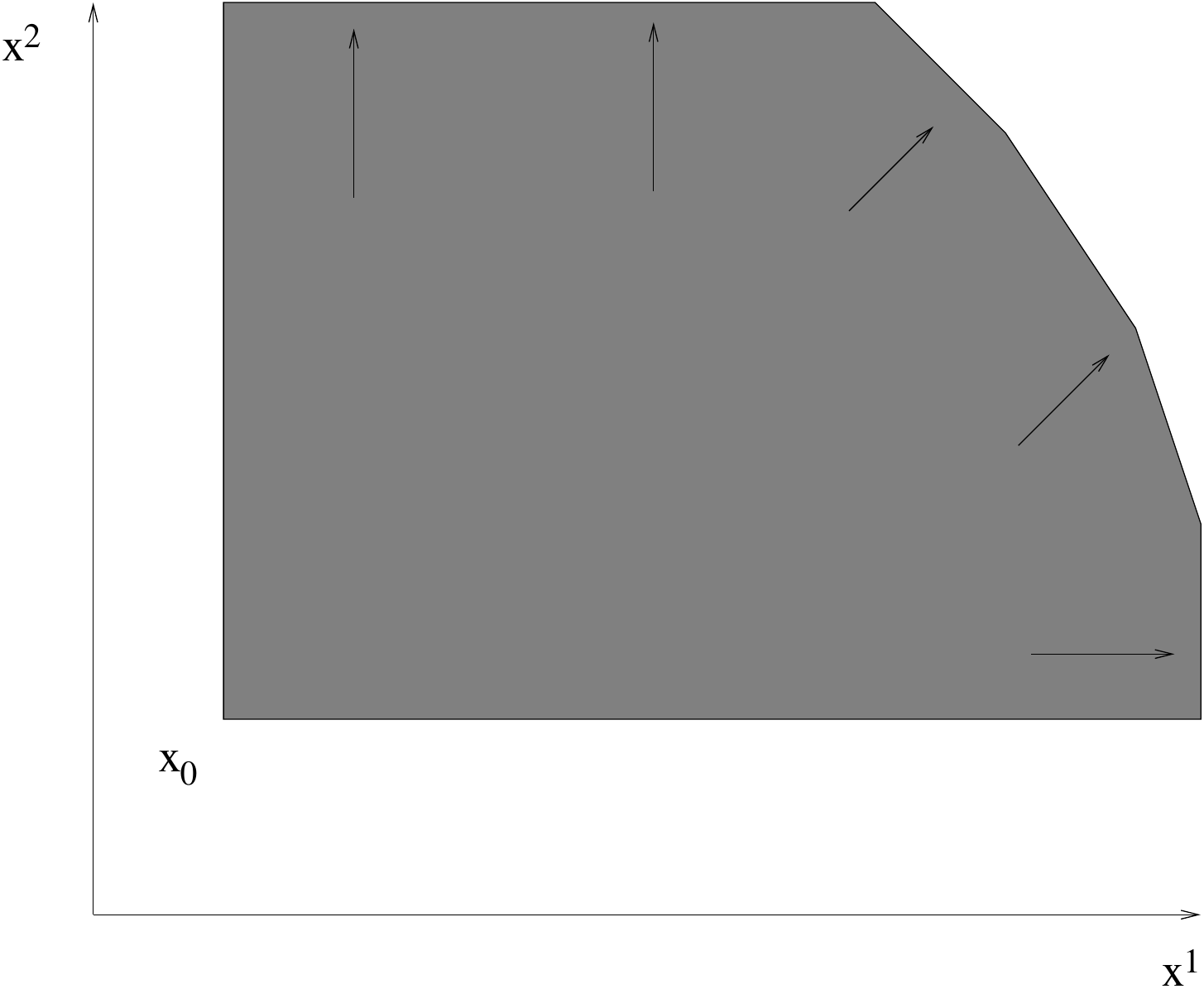}} \hspace{0.5cm}
    \subfigure[\label{fig2}$\mathcal{A}=x_0 + Q_1^+$ .]{
      \includegraphics[width=6cm]{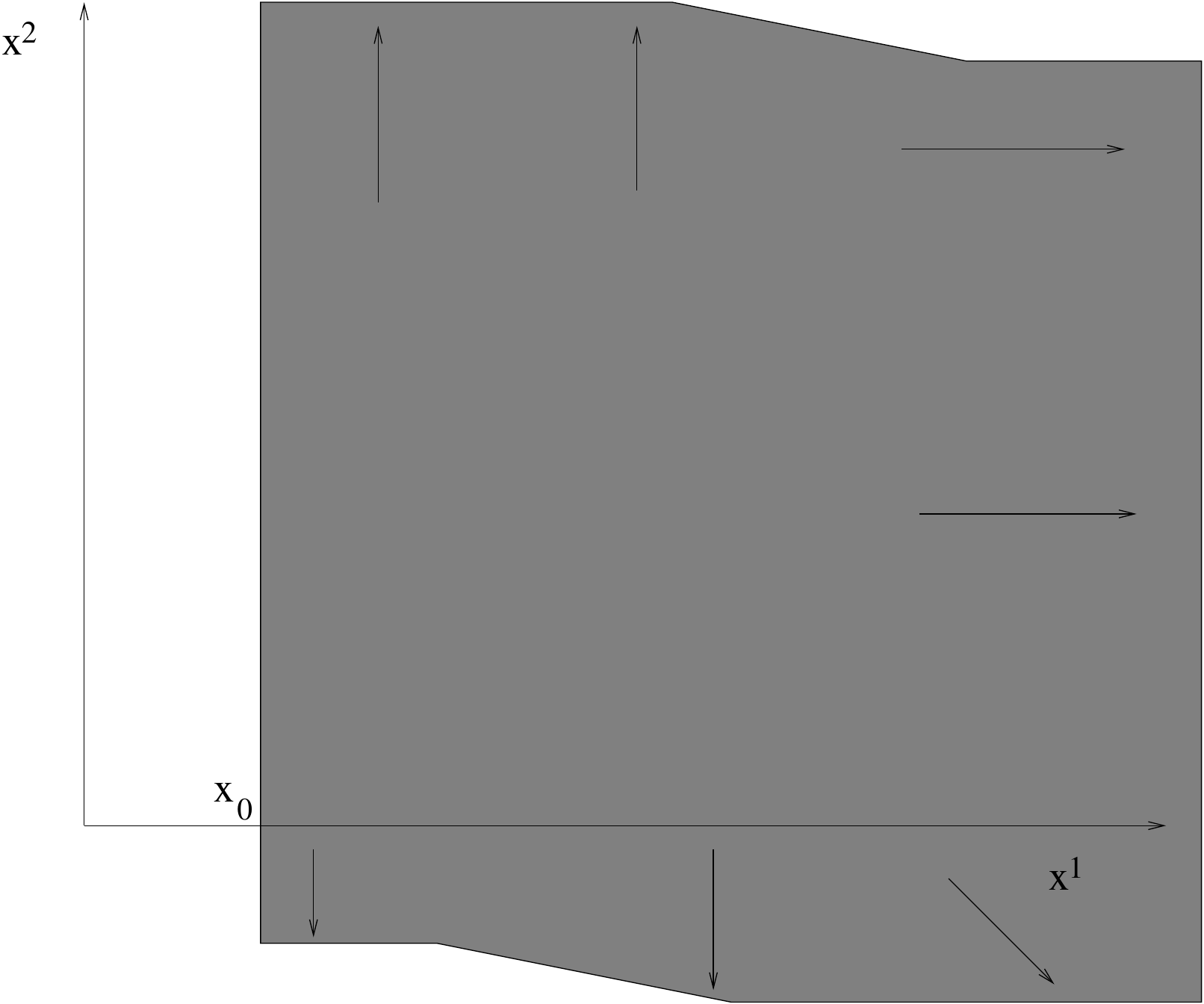}} \hspace{0.5cm}
    \subfigure[\label{fig3}$\mathcal{A}=x_0 + B{\Re_+^d}$.]{
      \includegraphics[width=6cm]{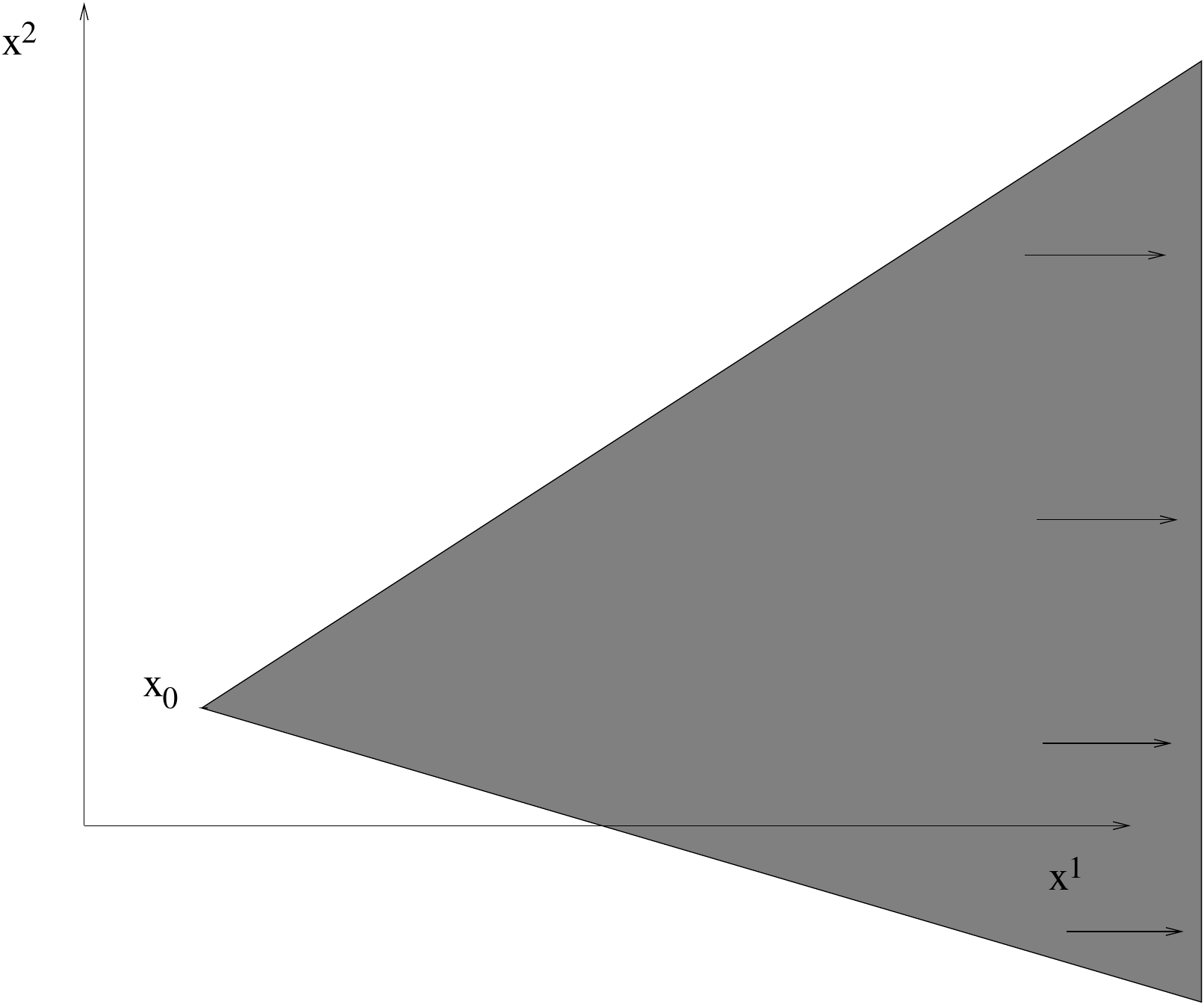}} \hspace{0.5cm}
    \subfigure[\label{fig4}$\mathcal{A}=x_0 + BQ_1^+$.]{
      \includegraphics[width=6cm]{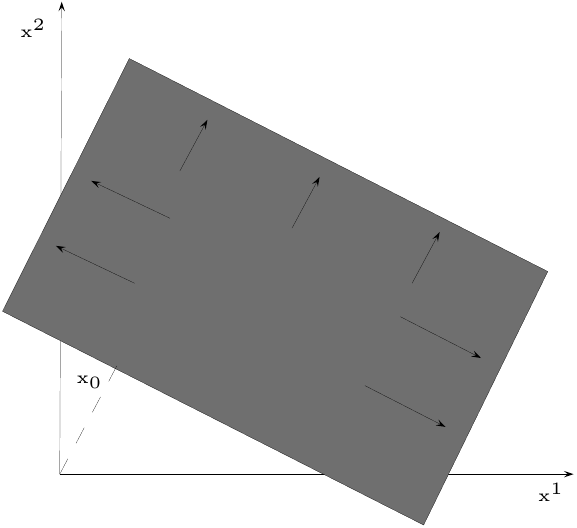}}
    \caption{\label{fig:shaded_region1} $\mathcal{A}$ is shown as
      shaded region ($d=2$).}
  \end{center}
\end{figure}

\begin{example}
  \label{positive:orthant_low_dim}
  For $0\leq d'\leq d$, let
$$Q_{d'}^+:=\{(x_1,x_2,\ldots,x_d)\in\Re^d|\,\,x_i\geq0\,\,\, \forall\,\,\, 0\leq i\leq d'\}.$$
Suppose we want to estimate $P[\bar{X}_n\in \mathcal{A}]$, where, now
$\mathcal{A}=x_0+Q_{d'}^+$ and $x_0$ is a given point in $\Re^d$ (see
Figure \ref{fig2}).  We proceed as in Example \ref{positive:orthant}.
In this case Equation (\ref{tail:generalized}) is
\begin{equation}
  \label{tail:special2}
  P[\bar{X}_n\in \mathcal{A}] = e^{-n\{\theta\cdot x_0 - \Lambda(\theta)\}}\int_{y\in Q_{d'}^+}e^{-\sqrt{n}(\theta\cdot y)}h_{n,\theta, x_0}(y)\,dy
\end{equation}
We now assume that $\theta^*_i>0, \,\,\forall i\leq d'$ and
$\theta^*_i=0\,\,\forall i>d'$

Dividing the right hand side of equation (\ref{tail:special2}) by
$\sqrt{n}\theta^*_1, \sqrt{n}\theta^*_2,\ldots, \sqrt{n}\theta^*_{d'}$
s and integrating out $y_{d'+1},y_{d'+2},\ldots,y_{d}$ we obtain
$$\frac{e^{n\left\{\Lambda(\theta^*)-\theta^* \cdot x_0\right\}}}{n^{\frac{d'}{2}}\theta^*_1\theta^*_2\cdots\theta^*_{d'}}\int_{y_i>0\forall i\leq d'}\left(\prod_{i=1}^{d'}\sqrt{n}\theta^*_i e^{-\sqrt{n}\theta^*_iy_i}\right)\left(\int_{y_i\in\mathbb{R}\forall d'<i\leq d}h_{n,\theta^*, x_0}(y)\prod_{i=d'+1}^{d}dy_i\right)\prod_{i=1}^{d'}dy_i,$$
which we can write as
$$\frac{e^{n\left\{\Lambda(\theta^*)-\theta^* \cdot x_0\right\}}}{n^{\frac{d'}{2}}\theta^*_1\theta^*_2\cdots\theta^*_{d'}}\int_{y_i>0\forall i\leq d'}\left(\prod_{i=1}^{d'}\sqrt{n}\theta^*_i e^{-\sqrt{n}\theta^*_iy_i}\right)\tilde{h}_{n,\theta^*, x_0}(y_1,y_2,\ldots,y_{d'})\prod_{i=1}^{d'}dy_i,$$
where $\tilde{h}_{n,\theta^*, x_0}(y_1,y_2,\ldots,y_{d'})$ is the
density function of $(Y^1,Y^2,\ldots,Y^{d'})$ under the measure
induced by $F_{\theta^*}$.  Thus, the problem reduces to that in
Example \ref{positive:orthant} with dimension $d'$ instead of $d$. In
this case,
$$c(n,\theta^*,x_0)=\frac{1}{n^{\frac{d'}{2}}\theta^*_1\theta^*_2\cdots\theta^*_{d'}}$$
and
$$\rho(n,\theta^*,x_0)(t_1,t_2,\ldots t_d)=\prod_{i=1}^{d'}\left(\frac{1}{1+\frac{\iota t_i}{\sqrt{n}\theta^*_i}}\right).$$
Thus, both the Assumptions \ref{cn:finite} and \ref{limit:chfn} hold
and we have
$$P[\bar{X}_n\in \mathcal{A}]\sim\frac{e^{n\left\{\Lambda(\theta^*)-\theta^* \cdot x_0\right\}}}
{(2\pi
  n)^{\frac{d'}{2}}\sqrt{\operatorname{det}(\Lambda^{\prime\prime}(\theta^{*}))}\theta^*_1\theta^*_2\cdots\theta^*_{d'}}.$$
Furthermore, the associated estimator has an asymptotically vanishing
relative error.
\end{example}

\begin{example}
  When $\mathcal{A}=x_0 + B{\Re_+^d}$ and $B$ a nonsingular matrix
  (see Figure \ref{fig3}), the problem can also be reduced to that
  considered in Example \ref{positive:orthant} by a simple change of
  variable.  Set $y=B^{-1}z$.  Then, it follows that for any $\theta$
$$c(n,\theta,x_0) = \text{det}(B)\int_{z\in\Re_+^d} \exp\{-\sqrt{n}(B^T\theta\cdot z)\}\,dz.$$
Now if we assume that all the $d$ components of $B^T\theta^*$ are
positive, then as in Example \ref{positive:orthant}, both the
Assumptions \ref{cn:finite} and \ref{limit:chfn} hold.

Similar analysis holds when $\mathcal{A}=x_0 + B{Q_{d'}^+}$, $(1 \leq
d' < d)$, and $B$ a nonsingular matrix. Then, simple change of
variable $y=B^{-1}z$ reduces the problem to that in
Example~\ref{positive:orthant_low_dim}.

\end{example}

\begin{figure}[ht]
  \begin{center}
    \subfigure[\label{subfig4}]{
      \includegraphics[width=6cm]{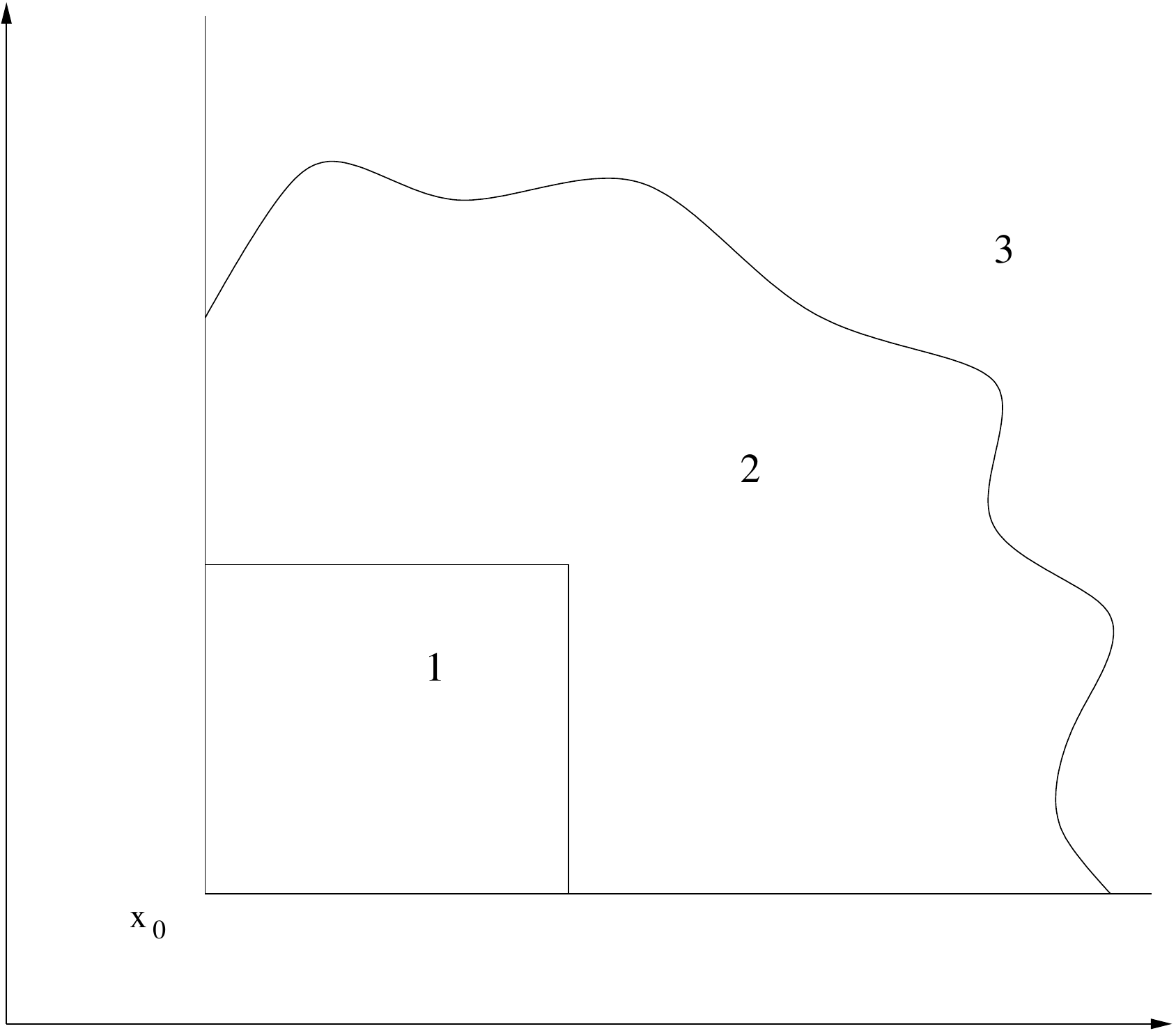}} \hspace{0.5cm}
    \subfigure[\label{subfig5}]{
      \includegraphics[width=6cm]{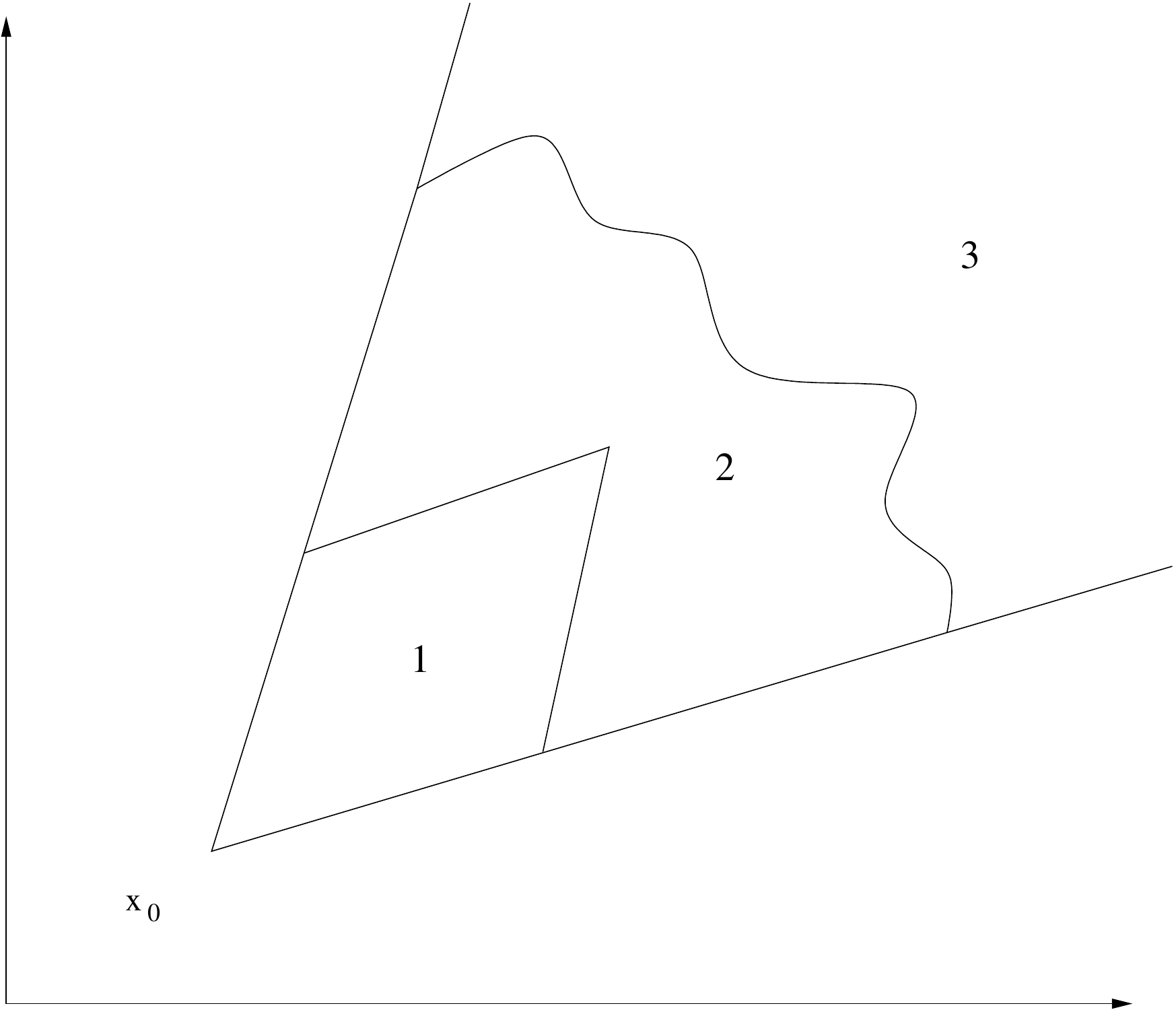}}
    \caption{\label{fig:shaded_region2}Set $\mathcal{A}^{(i)}$ is the
      region labeled $i$ ($i=1,2,3$,
      $\mathcal{A}^{(1)}\subset\mathcal{A}^{(2)}\subset\mathcal{A}^{(3)}$.)}
  \end{center}
\end{figure}

\begin{example}
  In above examples we have considered sets $\mathcal{A}$ which are
  unbounded.  In this example we show that similar analysis holds when
  the set $\mathcal{A}$ is bounded.  Consider the three increasing
  regions $(\mathcal{A}_i: i=1,2,3)$ as depicted in Figure~\ref{subfig4}.
  Here $\mathcal{A}_3$ corresponds to region $\mathcal{A}$ considered
  in Example~1.  $x_0$ is the common dominating point for all the
  three sets.  Again suppose that $\forall
  i=1,2,\ldots,d,\,\,\theta_i^*>0$.  Suppressing dependence on $x_0$
  and $\theta^*$, for $i=1,2$, let
$$c_n^{(i)}:=\int_{y\in\sqrt{n}(\mathcal{A}^{(i)}-x_0)} \exp\{-\sqrt{n}(\theta^*\cdot y)\}\,dy$$
and
$$\rho_n^{(i)}(t):=\frac{1}{c_n^{(i)}}\int_{y\in\sqrt{n}(\mathcal{A}^{(i)}-x_0)} \exp\{-\iota t\cdot y-\sqrt{n}(\theta^*\cdot y)\}\,dy.$$

If $\mathcal{A}^{(1)}$ is the $d$-dimensional rectangle given by
$\prod_{i}^d [x_0^i,x_0^i+D_i]$ then
$$c_n^{(1)}=\frac{(1-e^{-n\theta_1^*D_1})(1-e^{-n\theta_2^*D_2})\cdots(1-e^{-n\theta_d^*D_d})}{n^{\frac{d}{2}}\theta_1^*\theta_2^*\cdots\theta_d^*}$$
and
$$\rho_n^{(1)}(t_1,t_2,\ldots t_d)=\prod_{i=1}^{d}\left(\frac{1}{1+\frac{\iota t_i}{\sqrt{n}\theta_i^*}}\times\frac{1-e^{-n\theta_i^*D_i(1+\frac{\iota t_i}{\sqrt{n}\theta_i^*})}}{1-e^{-n\theta_i^*D_i}}\right).$$
Therefore, it follows that Assumption \ref{limit:chfn} holds for
$\mathcal{A}^{(1)}$.  Also note that,
\begin{eqnarray*}
  |\rho_n^{(2)}(t)-1| &\leq& \frac{1}{c_n^{(2)}}\int_{y\in\sqrt{n}(\mathcal{A}^{(2)}-x_0)} \exp\{-\sqrt{n}(\theta^*\cdot y)\}\left|e^{-\iota t\cdot y}-1\right|\,dy\\
  &\leq& \frac{1}{n^{\frac{d}{2}}c_n^{(1)}}\int_{z\in n(\mathcal{A}^{(2)}-x_0)} \exp\{-\theta^*\cdot z\}\left|e^{-\frac{\iota t\cdot z}{\sqrt{n}}}-1\right|\,dz\\
  &\leq& \frac{1}{n^{\frac{d}{2}}c_n^{(1)}}\int_{z\in \Re_+^d} \exp\{-\theta^*\cdot z\}\left|e^{-\frac{\iota t\cdot z}{\sqrt{n}}}-1\right|\,dz.\\
\end{eqnarray*}
Since the last integral converges to zero, it follows that Assumption
\ref{limit:chfn} holds for $\mathcal{A}^{(2)}$.  Similar analysis
carries over to sets as illustrated by Figure \ref{subfig5} under the
conditions as in Example 3.
\end{example}

In Example~1 we assumed that $\forall
i=1,2,\ldots,d,\,\,\theta_i^*>0$.  In many setting, this may not be
true but the problem can be easily transformed to be amenable to the
proposed algorithms. We illustrate this through the following
example. Essentially, in many cases where such a $\theta^*$ does not
exist, the problem can be transformed to a finite collection of
subproblems, each of which may then be solved using the proposed
methods.

\begin{example}
  Let $(X_i: i \geq 1)$ be a sequence of independent rv's with
  distribution same as $X=(Z_1,Z_2)$, where $Z_1$ and $Z_2$ are
  standard normal rvs with correlation $\rho$.  Suppose $\mathcal{A}
  := (a,b) + \Re^2_+$, that is $\mathcal{A} := \{(z_1,z_2)|z_1\geq
  a\,\, \text{and}\,\, z_2\geq b\}$.  Solving
  $\Lambda'(\theta_1,\theta_2) = (a,b)$ we get
$$\theta_1^*=\frac{a-\rho b}{1-\rho^2}\,\,\,\, \text{and}\,\,\,\, \theta_2^*=\frac{b-\rho a}{1-\rho^2}$$
Thus, if $\min\{\frac{a}{b},\frac{b}{a}\}>\rho$ we have both
$\theta_1^*$ and $\theta_2^*$ positive, and we are in situation of
Example \ref{positive:orthant}.  Suppose $\frac{b}{a}<\rho$ so that
$\theta_2^*<0$.  Then making the change of variable $Z_3=-Z_2$ we have
$$P[\bar{Z_1}\geq a, \bar{Z_2} \geq b]=P[\bar{Z_1}\geq a] - P[\bar{Z_1}\geq a, \bar{Z_3} \geq -b].$$
Now for estimating the second probability we have both $\theta_1^*$
and $\theta_2^*$ positive.  Similarly, the first probability is easily
estimated using the proposed algorithm.

However, note that if $(a,b)$ lies on $\{(z_1,z_2)|z_1=\rho z_2\,\,
\text{or}\,\,z_2=\rho z_1\}$ we have one of $\theta_1^*$ or
$\theta_2^*$ zero, and consequently $c(n,\theta_1^*,\theta_2^*,a,b)$
is infinite. The proposed algorithms may need to be modified to handle
such situations, however its not clear if simple adjustment to our
algorithm will result in the asymptotically vanishing relative error
property. We further discuss restrictions to our approach in Section
6.
\end{example}

\subsection{Estimating expected overshoot}

The methodology developed previously to estimate the tail probability
$P(\bar{X}_{n}\in \mathcal{A})$ can be extended to estimate $E\lbrack
\bar{X}_{n}^{\alpha} \,\vert \bar{X}_{n} \in \mathcal{A} \rbrack $ for
$\alpha \in (\mathbb{Z}_{+} - \{0\})^d$. We illustrate this in a single
dimension setting ($d=1$) for $\alpha=1$, and $\mathcal{A}= (x_0,
\infty)$ for $x_0 > EX_{i}$.

Let $S_n =\sum_{i=1}^n X_i$. In finance and in insurance one is often
interested in estimating $E \lbrack (S_{n} - nx_0) \vert S_{n} > nx_0
\rbrack$, which is known as the expected overshoot or the peak over
threshold. As we have an efficient estimator for $P(\bar{X}_{n} >
x_0)$, the problem of efficiently estimating $E \lbrack S_{n} \vert
S_{n} > nx_0 \rbrack$ is equivalent to that of efficiently estimating
$E \lbrack (S_{n} - nx_0)I(S_{n} > nx_0)\rbrack$.  Note that
\[E\lbrack ((S_{n} - nx_0) I(S_{n} > nx_0) \rbrack = \sqrt{n} E\lbrack
Y_{n}I(Y_{n}> 0) \rbrack , \] where $Y_{n} = \sqrt{n}(\bar{X}_{n} -
x_0 )$. Using (\ref{tail:generalized}) we get
\begin{equation}
  \label{eqn:overshoot1}
  E\lbrack Y_{n}I(Y_{n}> 0) \rbrack = e^{-n\{\theta^{*}\cdot x_0 -
    \Lambda(\theta^{*})\}}\int_{0}^{\infty} y\,e^{-\sqrt{n}(\theta^{*}\cdot
    y)}h_{n,\theta^{*}, x_0}(y)\,dy ,
\end{equation}
where recall that $\theta^*\in\Theta$ is a solution to
$\Lambda'(\theta) = x_0$ and $h_{n,\theta^{*}, x_0}(y)$ is the density
of $Y_{n}$ when each $X_{i}$ has distribution $F_{\theta^*}$. Define
$$\tilde{c}(n,\theta^*) = \int^{\infty}_{0} y\, \exp\{-\sqrt{n}(\theta^*\cdot y)\}\,dy= (n\, {\theta^{*}}^{2})^{-1}$$
Hence, $\forall n$, $\tilde{c}(n,\theta^*)<\infty$. The right hand
side of \eqref{eqn:overshoot1} may be re-expressed as
\begin{equation}
  \label{eqn:parseval}
  \tilde{c}(n,\theta^*) e^{-n\{\theta^*\cdot x_0
    - \Lambda(\theta^*)\}}\int^{\infty}_{0} \tilde{r}_{n,\theta^*}(y)h_{n,\theta^*,x_0}(y)\,dy\,
\end{equation}
where,
\begin{equation}
  \tilde{r}_{n,\theta^*}(y) = \left\{
    \begin{array}{lr}
      \frac{y\,\exp\{-\sqrt{n}(\theta^*\cdot y)\}}{\tilde{c}(n,\theta^*)} & \text{when}\,\, y > 0 \\
      0 & \text{otherwise}
    \end{array}
  \right.
\end{equation}
is a density in $\Re_+$.

Let $\tilde{\rho}_{n,\theta^*}(t)$ denote the complex conjugate of the
characteristic function of $\tilde{r}_{n,\theta^*}(y)$. By simple
calculations, it follows that
\[ \tilde{\rho}_{n,\theta^*}(t) =
\frac{1}{1-\frac{t^{2}}{n{\theta^{*}}^{2}}-\frac{2\iota\,t}{\sqrt{n}\theta^{*}}}
\,\, , \] and
$\underset{n\rightarrow\infty}{\lim}\tilde{\rho}_{n,\theta^*}(t) = 1.$
Then, repeating the analysis for the tail probability, analogously to
(\ref{final}), we see that \eqref{eqn:parseval} equals
\[\frac{\tilde{c}(n,\theta^*) e^{-n\{\theta^*\cdot x_0 -
    \Lambda(\theta^*)\}}}{\sqrt{2\pi\,\Lambda^{\prime\prime}(\theta^{*})}}\int^{\infty}_{0}\tilde{\rho}_{n,\theta^*}
(A(\theta^*)v)\psi(n^{-\frac{1}{2}}A(\theta^*)v,\theta^*,n)\phi(v)\,dv.\]

As in Proposition~\ref{asymptotic3}, we can see that
\[
E\lbrack (S_{n} - nx_0 ) I(S_{n} > nx_0 ) \rbrack \sim
\left(\frac{n}{2\pi}\right)^{\frac{1}{2}} \frac{\tilde{c}(n,\theta^*)
  e^{-n\{\theta^*\cdot x_0 -
    \Lambda(\theta^*)\}}}{\sqrt{\operatorname{det}(\Lambda^{\prime\prime}(\theta^{*}))}}
= \left(\frac{1}{2\pi n}\right)^{\frac{1}{2}} \frac{
  e^{-n\{\theta^*\cdot x_0 - \Lambda(\theta^*)\}}}{
  {\theta^*}^2\sqrt{\operatorname{det}(\Lambda^{\prime\prime}(\theta^{*}))}},
\]
so that
\[
\frac{E\lbrack (S_{n} - nx_0 ) I(S_{n} > nx_0)\rbrack}{P\lbrack S_{n}
  > nx_0 \rbrack} \sim \frac{1}{\theta^{*}}.
\]

Using analysis identical to that in Theorem~\ref{mainthm2}, it follows
that the resulting unbiased estimator of $E\lbrack (S_{n} - nx_0 )
I(S_{n} > nx_0) \rbrack$ (when density $g_n$ is used) has an
asymptotically vanishing relative error.

The above analysis can be easily extended to prove similar results for
the case of $X_{i} \in \Re^{d}$ and $\alpha$ a vector of positive integers.

\section{Numerical Experiments}
\subsection{Choice of parameters of IS density}
To implement the proposed method, the user must first specify the
parameters of the IS density $g_n$ appropriately.  In this subsection
we indicate how this may be done in practice.  All the user needs is
to identify a sequence $\{s_n\}_{n=1}^\infty$ satisfying the three
properties listed in Subsection \ref{proposedIS}.  Once
$\{s_n\}_{n=1}^\infty$ is specified, arriving at appropriate $\alpha$,
$a_n$, and $b_n$ is straightforward (see discussion before Theorem
\ref{mainthm0}; Finding $A(\theta^*)$, $\kappa_{max}$ and
$\kappa_{min}$ are one time computations and can be efficiently done
using MATLAB or MATHEMATICA).

Clearly for any $\epsilon\in(0,1)$, $s_n:=\frac{1}{n^\epsilon}$
satisfies properties 1 and 2.  To see that property 3 also holds, note
that
$$1-|\varphi_{\theta^*}(t)|^2=\int_{x\in\Re^d}(1-\cos(t\cdot x))d\tilde{F}_{\theta^*}(x),$$
where $\tilde{F}_{\theta^*}(x)$ is the symmetrization of
$F_{\theta^*}(x)$ (if $G$ is the distribution function of random
vector $Y$ then symmetrization of $G$, denoted $\tilde{G}$, is the
distribution function of the random vector $Y+Z$, where $Z$ has same
distribution as $-Y$).  Since
$$\frac{(t\cdot x)^2}{2!}-\frac{(t\cdot x)^4}{4!}\leq 1-\cos(t\cdot x)\leq\frac{(t\cdot x)^2}{2!},$$
it follows that there exist a neighborhood $U \subset \Re^d$ of origin
and positive constants $c$ and $C$, such that
$$c|t|^2\leq  1-|\varphi_{\theta^*}(t)|^2\leq C|t|^2$$
for all $t\in U$.  This in turn implies that there is a neighborhood
$V\subset \Re$ of zero and positive constants $c,C,c_1$ and $C_1$ such
that
$$cx^2\leq h(x)\leq Cx^2$$
and
$$c_1\sqrt{x}\leq h_1(x)\leq C_1\sqrt{x}$$
for all $x\in V$.  Therefore
$\sqrt{n}h_1(s_n)=\sqrt{n}h_1(n^{-\epsilon})\geq c
n^{\frac{1}{2}-\frac{\epsilon}{2}}\rightarrow \infty$ for any
$\epsilon<1$.

One may choose $\epsilon$ close to $1$ so that $\sqrt{n}h_1(s_n)$
grows slowly.  Then, since
$a_n=\sqrt{n}\delta_2(n)=\sqrt{\kappa_{max}}\sqrt{n}h_1(s_n)$, $a_n$
can be taken approximately a constant over a specified range of
variation of $n$.  Also since $p_n=b_n\times
IG\left(\frac{d}{2},\frac{a_n^2}{2}\right)$ is what one uses for
simulating from $g_n$, and $p_n\uparrow 1$, in practice for reasonable
values of $n$, one may take $p_n$ as a constant close to 1.  In our
numerical experiment below, parameters for $g_n$ are chosen using
these simple guidelines.

\subsection{Estimation of probability density function of
  $\bar{X}_{n}$}
We first use the proposed method to estimate the probability density
function of $\bar{X}_{n}$ for the case where sequence of random
variables $(X_{i}: i \geq 1)$ are independent and identically
exponentially distributed with mean 1. Then the sum has a known gamma
density function facilitating comparison of the estimated value to the true value.
 The
density function estimates using the proposed  method (referred to as  SP-IS
method)
are evaluated for $n =
30$, $a_{n} = 2, \alpha = 2$ and $p_{n} = 0.9$ (the algorithm
performance was observed to be relatively insensitive to small
perturbations in these values) based on $N$ generated samples.
Table~\ref{table:tabdense} shows the comparison of our method with the
conditional Monte Carlo (CMC) method proposed in Asmussen and Glynn
(2008) (pg. 145-146) for estimating the density function of
$\bar{X}_{n}$ at a few values.  As discussed in Asmussen and Glynn
(2008), the CMC estimates are given by an average of $N$ independent
samples of $nf(x-S_{n-1})$, where $S_{n-1}$ is generated by sampling
$(X_1, \ldots, X_{n-1})$ using their original density function $f$.
Figure~\ref{figure:denseIS} shows this comparison graphically over
a wider range of density function values. As may be expected, the
proposed method provides an estimator with much smaller variance
compared to the CMC method.

\begin{table}[htb]
  \centering
  \begin{tabular}{|c|c|c|c|c|c|}
    \hline
    $x$ & True value & SP-IS & Sample   & CMC      & Sample \\
    &            & estimate  & variance & estimate & variance \\
    \hline
    1.0 & 2.179 & 2.185 & 0.431 & 2.360 & 31.387 \\
    1.5 & 0.085 & 0.087 & 4.946 $\times 10^{-4}$ & 0.067 & 0.478 \\
    2.0 & 1.094 $\times10^{-4}$ & $1.105 \times 10^{-4}$ & $1.066 \times 10^{-9}$ & $7.342 \times 10^{-7}$ & $3.341 \times 10^{-1}$\\
    \hline
  \end{tabular}
  \caption{True density function and its estimates  using the proposed (SP-IS) method and the conditional Monte Carlo (CMC) for an average of 30 independent exponentially distributed mean = 1 random variables. For $x = 1.0$ and $1.5$, the number of generated samples $N = 1000$ in both the methods, and for $x=2.0$, $N=10,000$.}
  \label{table:tabdense}
\end{table}

\begin{figure}[htb]
\vspace{-10pt}
\begin{center}
\includegraphics[width=1.0\textwidth]{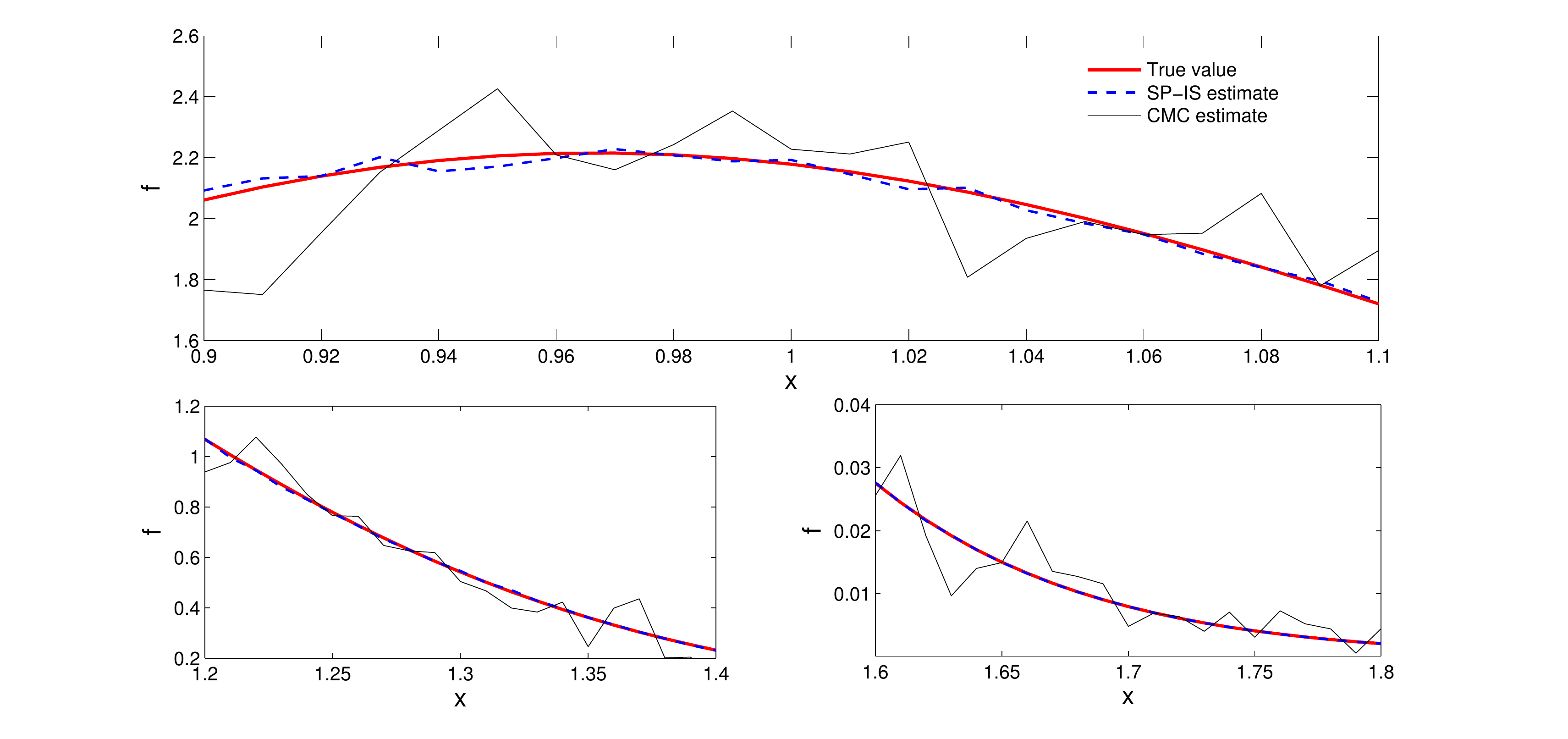}
\end{center}
\caption{True density function and its estimates  using the proposed (SP-IS) method and the conditional Monte Carlo (CMC) for an average of 30 independent exponentially distributed mean = 1 random variables. This plot illustrates the performance of the two methods over wide range of $x$ values. In both simulations $N=1,000$ at each point. 
 \label{figure:denseIS}}
\vspace{-20pt}
\end{figure}

\subsection{Comparison with independent exponential twisting approach}

We consider a simple numerical experiment in dimension $d=3$ to
compare efficiency of the proposed method with the one involving state
independent exponential twisting proposed by Sadowsky and Bucklew
(1990).  We consider a sequence of random vectors $(X_i, Y_i , Z_i: i
\geq 1)$ that are independent and identically distributed as follows:
Let $E_1,E_2,E_3$ be iid exponentially distributed with mean 1.
Define rvs $X$, $Y$ and $Z$ as
$$X=\frac{1}{2}\left(E_1+E_2\right)$$
$$Y=\frac{1}{2}\left(E_2+E_3\right)$$
$$Z=\frac{1}{2}\left(E_3+E_1\right)$$

Each $(X_i,Y_i,Z_i)$ for $i=1,2,\ldots,n$ has the same distribution as
$(X,Y,Z)$.  We estimate the probability $P(\bar{X}_n \geq x,\bar{Y}_n
\geq y,\bar{Z}_n \geq z)$ for $x=1.4$, $y=1.5$ and $z=1.4$ and
different values of $n$.  Table \ref{table:taboet} below reports the
estimates based on $N$ generated samples.  $c_n$ denotes the exact
asymptotic (the saddle point estimate) corresponding to the
probability. These differ substantially from the estimated probability
values, emphasizing the inaccuracy of $c_n$ even for reasonably large
values of $n$, and thus motivating simulation as a tool for accurate
estimation of the associated rare probabilities.

In these experiments we set $a_n=2$, $\alpha=3$ and $p_n=0.95$.  We
also report the variance reduction achieved by the proposed method
over the one proposed by Sadowsky and Bucklew (1990).  This is
substantial and it increases with increasing $n$.

\begin{table}[htb]
  \begin{center}
    \centering
    \caption{Comparison of the proposed methodology (SP-IS) with
      optimal state independent exponential twisting (OET).  In second
      and third columns we report the $95\%$ confidence intervals for
      the tail probability under SP-IS and OET respectively. }
    \label{table:taboet}
    \begin{tabular}{|r|l|l||l|}
      \hline
      n=10   &  &  &$c_n=0.0122562$\\
      \hline
      N & OET & SP-IS & Variance reduction\\
      \hline
      1000 & $(2.391 \pm 0.494)\times10^{-3}$ & $(2.492 \pm 0.211)\times10^{-3}$ & $5.48$ \\
      \hline
      10000 & $(2.546 \pm 0.163)\times10^{-3}$ & $(2.478 \pm 0.073)\times10^{-3}$ & $4.98$ \\
      \hline
      100000 & $(2.503 \pm 0.05)\times10^{-3}$ & $(2.479 \pm 0.024)\times10^{-3}$ & $4.34$ \\
      \hline
      n=20 & & &$c_n=4.490\times10^{-4}$\\
      \hline
      N & OET & SP-IS & Variance reduction\\
      \hline
      1000 & $(1.621 \pm 0.373)\times10^{-4}$ & $(1.383 \pm 0.102)\times10^{-4}$ &  $13.37$\\
      \hline
      10000 & $(1.507 \pm 0.118)\times10^{-4}$ & $(1.513 \pm 0.029)\times10^{-4}$ &  $16.55$\\
      \hline
      100000 & $(1.506 \pm 0.037)\times10^{-4}$ & $(1.474 \pm 0.009)\times10^{-4}$ &  $16.90$ \\
      \hline
      n=40 & & &$c_n=1.704\times10^{-6}$\\
      \hline
      N & OET & SP-IS & Variance reduction\\
      \hline
      1000 & $(7.349 \pm 2.346)\times10^{-7}$ & $(8.309 \pm 0.364)\times10^{-7}$ & $41.53$ \\
      \hline
      10000 & $(7.77 \pm 0.757)\times10^{-7}$ & $(8.186 \pm 0.115)\times10^{-7}$ & $43.33$ \\
      \hline
      100000 & $(8.039 \pm 0.255)\times10^{-7}$ & $(8.181 \pm 0.037)\times10^{-7}$ & $47.50$ \\
      \hline
      n=60 & & &$c_n=9.960\times10^{-9}$\\
      \hline
      N & OET & SP-IS & Variance reduction\\
      \hline
      1000 & $(5.411 \pm 2.051)\times10^{-9}$ & $(5.869 \pm 0.257)\times10^{-9}$ &  $63.69$\\
      \hline
      10000 & $(5.734 \pm 0.668)\times10^{-9}$ & $(5.632 \pm 0.071)\times10^{-9}$ &  $88.52$\\
      \hline
      100000 & $(5.666 \pm 0.214)\times10^{-9}$ & $(5.651 \pm 0.023)\times10^{-9}$ &  $86.57$\\
      \hline
      n=80 & & &$c_n=6.946\times10^{-11}$\\
      \hline
      N & OET & SP-IS & Variance reduction\\
      \hline
      1000 & $(4.101 \pm 1.664)\times10^{-11}$ & $(4.337 \pm 0.181)\times10^{-11}$ & $84.52$ \\
      \hline
      10000 & $(4.615 \pm 0.622)\times10^{-11}$ & $(4.401 \pm 0.059)\times10^{-11}$ & $111.14$ \\
      \hline
      100000 & $(4.343 \pm 0.187)\times10^{-11}$ & $(4.381 \pm 0.018)\times10^{-11}$ & $107.93$ \\
      \hline
      n=100 & & &$c_n=5.336\times10^{-13}$\\
      \hline
      N & OET & SP-IS & Variance reduction\\
      \hline
      1000 & $(3.676 \pm 1.478)\times10^{-13}$  & $(3.618 \pm 0.146)\times10^{-13}$ & $102.48$ \\
      \hline
      10000 & $(3.923 \pm 0.533)\times10^{-13}$ & $(3.637 \pm 0.049)\times10^{-13}$ & $118.32$ \\
      \hline
      100000 & $(3.546 \pm 0.172)\times10^{-13}$ & $(3.609 \pm 0.016)\times10^{-13}$ & $115.56$ \\
      \hline
    \end{tabular}
  \end{center}
\end{table}

\subsection{Comparison with state dependent exponential twisting}

We compare the efficiency of SP-IS  method  for estimating the tail probability $P(\bar{X}_{n} \in
\mathcal{A})$ with the optimal state dependent exponential twisting
method proposed by \cite{blan:gly} (referred to as BGL
method).  
They restrict their analysis to convex sets $\mathcal{A}$ with twice
continuously differentiable boundary whereas SP-IS method is
applicable to sets that are affine transformations of the non-negative
orthants $\Re^d_+$. The two methods agree in the single dimension and
hence we compare them on a single dimension example.

For  a sequence of random variables $(X_{i}: i \geq 1)$ that
are independent and identically exponentially distributed with mean
1, $P(\bar{X}_{n} \geq 1.5)$ is estimated  for different
values of $n$. Table~\ref{table:tab1d} reports the estimates based on
different $N$ generated samples. In this experiment, $a_{n} = 2,
\alpha = 2$ and $p_{n} = 0.9$ for SP-IS method. BGL method is
implemented as per \cite{blan:gly} as follows: first $X_{1}$ is generated using an
exponentially twisted distribution with mean $x_{0} = 1.5$. At each
next step, the exponential twisting coefficient in the distribution
used to generate $X_{k+1}$ is recomputed such that mean of the
distribution is $\frac{nx_{0} - \sum_{i=1}^{k}X_i}{n-k}$. The
exponential twisting is dynamically updated until the generated
$\sum_{i=1}^{k}X_i \geq nx_{0}$ at which point we stop the importance
sampling and sample rest of $n-k$ values with the original
distribution. In the other case, if distance to the boundary $nx_{0} -
\sum_{i=1}^{k}X_i$ is sufficiently large relative to remaining time
horizon $n - k$ $\big(\frac{nx_{0} - \sum_{i=1}^{k}X_i}{n - k} \geq
2\,x_{0} \big)$, then we generate the next $n-k$ samples with
exponentially twisted distribution with mean $\frac{nx_{0} -
  \sum_{i=1}^{k}X_i}{n - k}$.

\begin{table}[htb]
  \centering
  \begin{tabular}{|p{0.5cm}|p{0.5cm}|p{2.5cm}|p{2.0cm}|p{0.6cm}|p{2.0cm}|p{0.6cm}|c|p{0.6cm}|p{0.4cm}|}
    \hline
    n  & N      & True value             & BGL                   & CoV  & SP-IS                & CoV  & VR    & \multicolumn{2}{|c|}{CT}\\ \cline{9-10}
    &        &(exact asymptotic $c_n$)&                       &      &                      &      &       & BGL  & SP-IS\\ \hline
    \hline
    & $10^3$ &                        & 9.276$\times 10^{-4}$  & 1.41 & 9.055$\times 10^{-4}$ & 0.32 & 20.38 &      &\\
 50 & $10^4$ & 9.039$\times 10^{-4}$   & 9.127$\times 10^{-4}$  & 1.41 & 9.036$\times 10^{-4}$ & 0.32 & 19.77 & 7.5  & 0.9\\
    & $10^5$ &(9.992$\times 10^{-4}$)  & 9.036$\times 10^{-4}$  & 1.41 & 9.038$\times 10^{-4}$ & 0.32 & 19.13 &      &\\ \hline

    & $10^3$ &                        & 5.936$\times 10^{-6}$  & 1.44 & 5.932$\times 10^{-6}$ & 0.28 & 25.84 &      &\\
100 & $10^4$ & 5.924$\times 10^{-6}$   & 5.913$\times 10^{-6}$  & 1.45 & 5.923$\times 10^{-6}$ & 0.29 & 24.54 & 15.4 & 0.9\\
    & $10^5$ & (6.261$\times 10^{-6}$) & 5.928$\times 10^{-6}$  & 1.44 & 5.921$\times 10^{-6}$ & 0.29 & 24.20 &      &\\ \hline

    & $10^3$ &                        & 3.355$\times 10^{-10}$ & 1.48 & 3.378$\times 10^{-10}$ & 0.28 & 25.83 &      &\\
200 & $10^4$ & 3.371$\times 10^{-10}$  & 3.381$\times 10^{-10}$ & 1.46 & 3.368$\times 10^{-10}$ & 0.29 & 26.17 & 32.0 & 0.9\\
    & $10^5$ &(3.473$\times 10^{-10}$) & 3.370$\times 10^{-10}$ & 1.46 & 3.374$\times 10^{-10}$ & 0.28 & 26.92 &      &\\ \hline

    & $10^3$ &                        & 2.169$\times 10^{-14}$ & 1.46 & 2.180$\times 10^{-14}$ & 0.29 & 26.48 &      &\\
300 & $10^4$ & 2.176$\times 10^{-14}$  & 2.180$\times 10^{-14}$ & 1.47 & 2.175$\times 10^{-14}$ & 0.28 & 27.76 & 48.0 & 0.9\\
    & $10^5$ & (2.226$\times 10^{-14}$)& 2.173$\times 10^{-14}$ & 1.47 & 2.179$\times 10^{-14}$ & 0.28 & 27.89 &      &\\
    \hline
  \end{tabular}
  \caption{SP-IS method has a decreasing coefficient of variation (CoV) and it provides increasing variance reduction (VR) over the optimal state dependent exponential twisting (BGL) method. Computation time per sample (CT), reported in micro seconds, increases with $n$ for BGL method whereas it remains constant for SP-IS method.}
  \label{table:tab1d}
\end{table}

In this example, the true value of tail probability for different
values of $n$ is calculated using approximation of gamma density
function available in MATLAB. Variance reduction achieved by SP-IS
method over BGL method is reported. This increases with increasing
$n$. In addition, we note that the computation time per sample for BGL
method increases with $n$ whereas it remains constant for the SP-IS
method. Table~\ref{table:tab1d} shows
that the exact asymptotic  $c_n$ can differ significantly from the estimated value of the probability. As shown in Table~\ref{table:taboet}, this difference can be far more significant in multi-dimension settings, thus emphasizing the need for simulation
despite the existence of asymptotics for the rare quantities considered.

\section{Conclusions and Direction for Further Research}

In this paper we considered the rare event problem of efficient
estimation of the density function of the average of iid light tailed
random vectors evaluated away from their mean, and the tail
probability that this average takes a large deviation. In a single
dimension setting we also considered the estimation problem of
expected overshoot associated with a sum of iid random variables
taking a large deviations.  We used the well known saddle point
representations for these performance measures and applied importance
sampling to develop provably efficient unbiased estimation algorithms
that significantly improve upon the performance of the existing
algorithms in literature and are simple to implement.

In this paper we combined rare event simulation with the classical
theory of saddle point based approximations for tail events.  We hope
that this approach spurs research towards efficient estimation of much
richer class of rare event problems where saddle point approximations
are well known or are easily developed.

Another direction that is important for further research involves
relaxing Assumptions \ref{cn:finite} or \ref{limit:chfn} in our
analysis.  Then, our IS estimators may not have asymptotically
vanishing relative error but may have bounded relative error.  We
illustrate this briefly through a simple example below.  Note that
many intricate asymptotics developed by Iltis \cite{iltis} for
estimating $P[\bar{X}_n\in \mathcal{A}]$ correspond to cases where
Assumptions \ref{cn:finite} or \ref{limit:chfn} may not hold.

\begin{example}
  Let $(X_i: i \geq 1)$ be a sequence of independent rv's with
  distribution same as $X=(Z_1,Z_2)$, where $Z_1$ and $Z_2$ are
  uncorrelated standard normal rvs.  Suppose $\mathcal{A} :=
  \{(z_1,z_2)|z_1\geq z_2^2 + a\}$ for some $a>0$ (see Figure
  \ref{fig0}).
  \begin{figure}[t]
    \begin{center}
      \includegraphics[width=6cm]{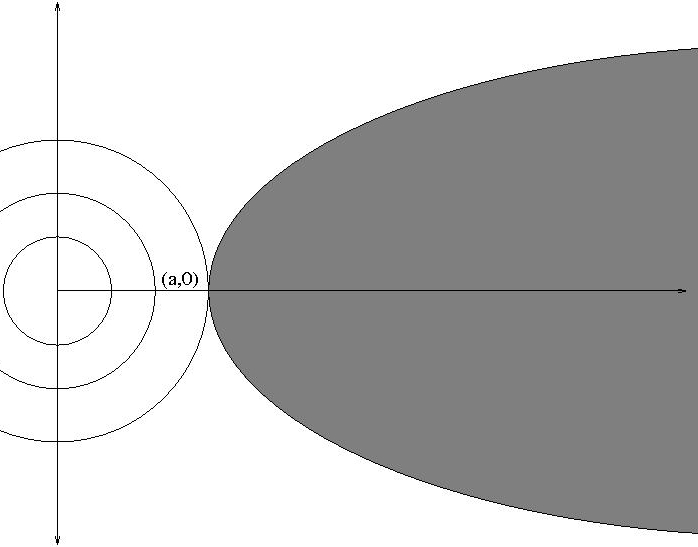}
      \caption{\label{fig0} $\mathcal{A} = \{(x^1,x^2)|x^1\geq (x^2)^2
        + a\}$.}
    \end{center}
  \end{figure}
  As $x_0$ we choose the point $(a,0)$ which is clearly the dominating
  point of the set $\mathcal{A}$.  Now for any $\theta_1>0$ and
  $\theta_2$ it can be shown that
$$c(n,\theta_1,\theta_2,a) = \int_{\{\sqrt{n}y_1\geq y_2^2\}}\exp\{-\sqrt{n}(\theta_1 y_1 + \theta_2 y_2)\}dy_1dy_2=\frac{\sqrt{\pi}\exp\{\frac{n\theta_2^2}{4\theta_1}\}}{\sqrt{n}\theta_1^{\frac{3}{2}}}.$$
Solving $\Lambda'(\theta_1,\theta_2) = (a,0)$ gives $\theta_1^*=a$ and
$\theta_2^*=0$.  Also
$$\rho_{n,\theta^*,x_0}(t) = \left(\frac{1}{1-\frac{\iota t_1}{a\sqrt{n}}}\right)^{\frac{3}{2}}\exp\left\{\frac{-t_2^2}{4(a-\frac{\iota t_1}{\sqrt{n}})}\right\}.$$
Therefore Assumption \ref{limit:chfn} fails to hold:
$$\lim_{n\rightarrow\infty}\rho_{n,\theta^*,x_0}(t)=\exp\left\{-\frac{t_2^2}{4a}\right\}.$$

\end{example}
Therefore, in this case the the family of estimator given by
(\ref{tail:estimator}) may not have asymptotically vanishing relative
error.  But, nevertheless, it can be shown to have bounded relative
error.  To see this, note that
$$\int_{v\in\Re^d}\rho_{x_0,\theta^*}(A(\theta^*)v)\phi(v)\,dv = \left(1 + \frac{1}{2a}\right)^{-\frac{1}{2}}$$
and
$$\int_{v\in\Re^d}\rho_{x_0,\theta^*}(A(\theta^*)v)^2\phi(v)\,dv = \left(1 + \frac{1}{a}\right)^{-\frac{1}{2}}.$$
(Here $\Lambda''(\theta)={1\,0 \choose 0\,1}$ for all $\theta$.  So
$A(\theta^*)={1\,0 \choose 0\,1}$.)  Also $\forall 1\leq i,j,k \leq d$
$$\int_{v\in\Re^d}v_iv_jv_k \rho_{x_0,\theta^*}(A(\theta^*)v)\phi(v)\,dv = 0 = \int_{v\in\Re^d}v_iv_jv_k \rho_{x_0,\theta^*}(A(\theta^*)v)^2\phi(v)\,dv .$$
Therefore as in Proposition \ref{asymptotic3}, it follows that
$$P[\bar{X}_n\in \mathcal{A}]\sim \frac{e^{-\frac{na^2}{2}}}{2\sqrt{\pi}\sqrt{n}a^{\frac{3}{2}}}\times\left(1 + \frac{1}{2a}\right)^{-\frac{1}{2}}.$$
Mimicking the proof of Theorem (\ref{mainthm2}) it can be established
that
$$Var_n\left[\hat {P}[\bar{X}_n\in \mathcal{A}]\right]\rightarrow \frac{1 + \frac{1}{2a}}{\sqrt{1 + \frac{1}{a}}}-1.$$

\appendix

\section{Proofs}
\begin{proof}(of Proposition \ref{asymptotic1})\\
Let $\zeta_3(\theta^*)=\Lambda'''(\theta^*)\star A(\theta^*)$.
We have
\begin{eqnarray*}
\left|\int_{v\in\Re^d}\psi(n^{-\frac{1}{2}}A(\theta^*)v,\theta^*,n)\phi(v)dv-1\right|&=&\left|\int_{v\in\Re^d}\{\psi(n^{-\frac{1}{2}}A(\theta^*)v,\theta^*,n)-1\}\phi(v)dv\right|\\
&=&\left|\int_{v\in\Re^d}\left\{\psi(n^{-\frac{1}{2}}A(\theta^*)v,\theta^*,n)-1-\frac{\zeta_3(\theta^*)}{6\sqrt{n}}\odot(\iota v)\right\}\phi(v)dv\right|\\
&\leq&\int_{v\in\Re^d}\left|\psi(n^{-\frac{1}{2}}A(\theta^*)v,\theta^*,n)-1-\frac{\zeta_3(\theta^*)}{6\sqrt{n}}\odot(\iota v)\right|\phi(v)dv\\&=&
\frac{1}{(2\pi)^{\frac{d}{2}}}(I_1+I_2)\,,
\end{eqnarray*}
where
$$I_1=\int_{|n^{-\frac{1}{2}}A(\theta^*)v|<\delta}\left|\exp\left\{n\times\eta(n^{-\frac{1}{2}}A(\theta^*)v,\theta^*)\right\}-1-n\frac{\Lambda'''(\theta^*)}{3!}\odot\left(\iota n^{-\frac{1}{2}}A(\theta^*)v\right)\right|\exp\left\{-\frac{v^2}{2}\right\}\,dv,$$
$$I_2=\int_{|n^{-\frac{1}{2}}A(\theta^*)v|\geq\delta}\left|\exp\left\{n\times\eta(n^{-\frac{1}{2}}A(\theta^*)v,\theta^*)\right\}-1-n\frac{\Lambda'''(\theta^*)}{3!}\odot\left(\iota n^{-\frac{1}{2}}A(\theta^*)v\right)\right|\exp\left\{-\frac{v^2}{2}\right\}\,dv.$$

We now discuss how the $\delta$ above may be selected.

Since $\eta'''$ is continuous, it follows from the three term Taylor series expansion,
$$\eta(v,\theta)=\eta(0 ,\theta)+\eta'(0 ,\theta)v+\frac{1}{2}(v)^T\eta''(0,\theta)v+\frac{1}{6}\eta'''(\tilde{v} ,\theta)\odot v$$
(where $\tilde{v}$ is between $v$ and the origin),  (\ref{property1}) and (\ref{property2})  that for any given $\epsilon$ we
can choose $\delta$ small enough so that
$$|\eta(v,\theta^*)-\frac{1}{3!}\eta'''(0,\theta^*)\odot v|\leq \epsilon(\kappa_{min})^{\frac{3}{2}}|v|^3\,\,\,\,\text{for}\,\,|v|<\delta,$$
or equivalently
\begin{equation}
\label{estimate1}
|\eta(v,\theta^*)-\frac{1}{3!}\Lambda'''(\theta^*)\odot (\iota v)|\leq \epsilon(\kappa_{min})^{\frac{3}{2}}|v|^3\,\,\,\,\text{for}\,\,|v|<\delta\,.
\end{equation}

Since
\begin{equation}
\label{estimate2}
\left|\frac{1}{3!}\Lambda'''(\theta^*)\odot (\iota v)\right|<\frac{1}{8}\kappa_{min}|v|^2
\end{equation}
and
\begin{equation}
\label{estimate3}
|\eta(v,\theta^*)|<\frac{1}{8}\kappa_{min}|v|^2
\end{equation}
for all $|v|$ sufficiently small, we choose $\delta$ so that (\ref{estimate2}) and (\ref{estimate3}) also hold for
$|v|<\delta$.

We apply Lemma (\ref{keylemma}) with
$$\lambda=n\times\eta\left(n^{-\frac{1}{2}}A(\theta^*)v,\theta^*\right)\,\,\,\text{and}\,\,\,\beta=n\frac{\Lambda'''(\theta^*)}{3!}\odot\left(\iota n^{-\frac{1}{2}}A(\theta^*)v\right).$$
Since $\frac{|\beta|^2}{2}=\frac{1}{n}P(v)$, where $P$ is a homogeneous polynomial
with coefficients independent of $n$ and for $|n^{-\frac{1}{2}}A(\theta^*)v|<\delta$ we have from (\ref{estimate3}),
(\ref{estimate2}) and (\ref{estimate1}),
respectively,

$$|\lambda|=n\left|\eta\left(n^{-\frac{1}{2}}A(\theta^*)v,\theta^*\right)\right|<n\frac{1}{8}\kappa_{min}|n^{-\frac{1}{2}}A(\theta^*)v|^2\leq\frac{1}{8}\kappa_{min}||A(\theta^*)||^2|v|^2=\frac{|v|^2}{8}\,\,,$$
$$|\beta|=n\left|\frac{1}{3!}\Lambda'''(\theta^*)\odot\left(\iota n^{-\frac{1}{2}}A(\theta^*)v\right)\right|<n\frac{1}{8}\kappa_{min}|n^{-\frac{1}{2}}A(\theta^*)v|^2\leq\frac{1}{8}\kappa_{min}||A(\theta^*)||^2|v|^2=\frac{|v|^2}{8}$$
and
$$|\lambda-\beta|=n\left|\eta\left(n^{-\frac{1}{2}}A(\theta^*)v,\theta^*\right)-\frac{1}{3!}\Lambda'''(\theta^*)\odot\left(\iota n^{-\frac{1}{2}}A(\theta^*)v\right)\right|<n\epsilon(\kappa_{min})^{\frac{3}{2}}|n^{-\frac{1}{2}}A(\theta^*)v|^3\leq\frac{\epsilon|v|^3}{\sqrt{n}}\,.$$
From Lemma (\ref{keylemma}) it now follows that the integrand in $I_1$ is dominated by
$$\exp\left\{\frac{v^2}{8}\right\}\times\left(\frac{\epsilon|v|^3}{\sqrt{n}}+\frac{1}{n}P(v)\right)\times \exp\left\{-\frac{v^2}{2}\right\}
=\exp\left\{-\frac{3v^2}{8}\right\}\left(\frac{\epsilon|v|^3}{\sqrt{n}}+\frac{1}{n}P(v)\right).$$
Since $\epsilon$ is arbitrary we have $I_1=o(n^{-\frac{1}{2}})$.

Next we have
\begin{eqnarray*}
I_2&\leq&\int_{|n^{-\frac{1}{2}}A(\theta^*)v|\geq\delta}\left|\exp\left\{-\frac{v^2}{2}\right\}\psi(n^{-\frac{1}{2}}A(\theta^*)v,\theta^*,n)\right|\,dv\\
&+&\int_{|n^{-\frac{1}{2}}A(\theta^*)v|\geq\delta}\left(1+\left|\frac{\zeta_3(\theta^*)\odot v}{6}\right|\right)\exp\left\{-\frac{v^2}{2}\right\}\,dv,\\
&=&\int_{|A(\theta^*)v|\geq\delta\sqrt{n}}\left|\varphi_{\theta^*}\left(n^{-\frac{1}{2}}A(\theta^*)v\right)\right|^n\,dv+\int_{|A(\theta^*)v|\geq\delta\sqrt{n}}\left(1+\left|\frac{\zeta_3(\theta^*)\odot v}{6}\right|\right)\exp\left\{-\frac{v^2}{2}\right\}\,dv.
\end{eqnarray*}
Let $q_{\delta}<1$ be such that $|\varphi_{\theta^*}(v)|<q_{\delta}$ for $|v|\geq \delta$.
Then we have
\begin{eqnarray*}
I_2&\leq& q_\delta^{n-\gamma}\int_{v\in\Re^d}\left|\varphi_{\theta^*}\left(n^{-\frac{1}{2}}A(\theta^*)v\right)\right|^\gamma\,dv+\int_{|A(\theta^*)v|\geq\delta\sqrt{n}}\left(1+\left|\frac{\zeta_3(\theta^*)\odot v}{6}\right|\right)\exp\left\{-\frac{v^2}{2}\right\}\,dv,\\
&=& q_{\delta}^{n-\gamma}n^{\frac{d}{2}}\sqrt{|\Lambda''(\theta^*)|}\int_{v\in\Re^d}\left|\varphi_{\theta^*}(u)\right|^\gamma\,du+\int_{|A(\theta^*)v|\geq\delta\sqrt{n}}\left(1+\left|\frac{\zeta_3(\theta^*)\odot v}{6}\right|\right)\exp\left\{-\frac{v^2}{2}\right\}\,dv.
\end{eqnarray*}
It follows that $I_2=o(n^{-\alpha})$ for any $\alpha$.
\end{proof}

\begin{proof} (of Theorem \ref{mainthm2})\\
The proof follows along the same line as proof of Theorem \ref{mainthm0}.
We write
$$\int_{v\in\Re^d}\frac{\rho_{n,\theta^*,x_0}^2(A(\theta^*)v)\psi^2(n^{-\frac{1}{2}}A(\theta^*)v,\theta^*,n)\phi^2(v)}{g_n(v)}\,dv=I_5+I_6$$
where
\begin{eqnarray*}
I_5&=&\int_{|v|<\delta_2(n)\sqrt{n}}\frac{\rho_{n,\theta^*,x_0}^2(A(\theta^*)v)\psi^2(n^{-\frac{1}{2}}A(\theta^*)v,\theta^*,n)\phi^2(v)}{g_n(v)}\,dv\\
   &=&\frac{1}{b_n}\int_{|v|<\delta_2(n)\sqrt{n}}\rho_{n,\theta^*,x_0}^2(A(\theta^*)v)\psi^2(n^{-\frac{1}{2}}A(\theta^*)v,\theta^*,n)\phi(v)\,dv.
\end{eqnarray*}
\begin{eqnarray*}
I_6&=&\int_{|v|\geq\delta_2(n)\sqrt{n}}\frac{\rho_{n,\theta^*,x_0}^2(A(\theta^*)v)\psi^2(n^{-\frac{1}{2}}A(\theta^*)v,\theta^*,n)\phi^2(v)}{g_n(v)}\,dv\\
   &=&\frac{1}{C_{n}}\int_{|v|\geq\delta_2(n)\sqrt{n}}\rho_{n,\theta^*,x_0}^2(A(\theta^*)v)|v|^\alpha\psi^2(n^{-\frac{1}{2}}A(\theta^*)v,\theta^*,n)\phi^2(v)\,dv.
\end{eqnarray*}
Now
\begin{eqnarray*}
|I_5 - 1|&=&\left|\frac{1}{b_n}\int_{|v|<\delta_2(n) \sqrt{n}}\rho_{n,\theta^*,x_0}^2(A(\theta^*)v)\psi^2(n^{-\frac{1}{2}}A(\theta^*)v,\theta^*,n)\phi(v)\,dv-1\right|\\
&\leq&\frac{1}{b_n}\left|\int_{|v|<\delta_2(n) \sqrt{n}}\rho_{n,\theta^*,x_0}^2(A(\theta^*)v)\left\{\psi^2(n^{-\frac{1}{2}}A(\theta^*)v,\theta^*,n)-1\right\}\phi(v)\,dv\right|+o(1)\\
&\leq&\frac{1}{b_n}\left|\int_{|v|<\delta_2(n) \sqrt{n}}\rho_{n,\theta^*,x_0}^2(A(\theta^*)v)\left\{\psi^2(n^{-\frac{1}{2}}A(\theta^*)v,\theta^*,n)-1-\frac{\zeta_3(\theta^*)}{3\sqrt{n}}\odot(\iota v)\right\}\phi(v)\,dv\right| + o(1)\\
&\leq&\frac{1}{b_n}\int_{|v|<\delta_2(n) \sqrt{n}}\left|\rho_{n,\theta^*,x_0}(A(\theta^*)v)\right|^2\left|\psi^2(n^{-\frac{1}{2}}A(\theta^*)v,\theta^*,n)-1-\frac{\zeta_3(\theta^*)}{3\sqrt{n}}\odot(\iota v)\right|\phi(v)\,dv + o(1)\\
&\leq&\frac{1}{b_n}\int_{|v|<\delta_2(n) \sqrt{n}}\left|\psi^2(n^{-\frac{1}{2}}A(\theta^*)v,\theta^*,n)-1-\frac{\zeta_3(\theta^*)}{3\sqrt{n}}\odot(\iota v)\right|\phi(v)\,dv + o(1).
\end{eqnarray*}
Now as in the case of Theorem \ref{mainthm0} we conclude that  $I_5= 1 + o(n^{-\frac{1}{2}})$.
Also, since
\begin{eqnarray*}
|I_6|&\leq&\frac{1}{C_{n}}\int_{|A(\theta^*)v|\geq\delta_2(n)\sqrt{n}}|v|^\alpha\left|\rho_{n,\theta^*,x_0}(A(\theta^*)v)\right|^2\left|\psi^2(n^{-\frac{1}{2}}A(\theta^*)v,\theta^*,n)\right|\phi^2(v)\,dv\\
     &\leq&\frac{1}{(2\pi)^d C_{n}}\int_{|A(\theta^*)v|\geq\delta_2(n)\sqrt{n}}|v|^\alpha\left|\exp\left\{-v^2\right\}\psi^2(n^{-\frac{1}{2}}A(\theta^*)v,\theta^*,n)\right|\,dv,
\end{eqnarray*}
we  conclude that $I_6\rightarrow0$ as $n\rightarrow\infty$ proving the theorem.
\end{proof}


%
%
%
%

\end{document}